\documentclass[a4paper,12pt]{amsart}
\usepackage{lipsum}
\usepackage{amssymb,amsthm}
\usepackage[foot]{amsaddr}
\usepackage{amsmath}
\usepackage{tikz}
\usetikzlibrary{quotes,angles}
\usepackage{graphicx}
\usepackage{subcaption}
\usepackage[shortlabels]{enumitem}
\usetikzlibrary{arrows}
\usepackage{float}
\usepackage{graphicx}
\usepackage{subcaption}
\usepackage{hyperref}
\usepackage[section]{placeins}
\graphicspath{{Spindle_cylinder_figs/}}
\usetikzlibrary{decorations.pathmorphing}
\tikzset{snake it/.style={decorate, decoration=snake}}

\makeatletter
\newcommand*{\rom}[1]{\expandafter\@slowromancap\romannumeral #1@}
\makeatother

\numberwithin{equation}{section}

\theoremstyle{plain}
\newtheorem{theorem}{Theorem}
\numberwithin{theorem}{section}

\theoremstyle{definition}
\newtheorem{definition}[theorem]{Definition}
\theoremstyle{remark}
\newtheorem{remark}[theorem]{Remark}

\theoremstyle{remark}

\theoremstyle{remark}

\newcommand{\smo}{\setminus \mathbf{0}}

\newcommand{\norm}[1]{\left\lVert#1\right\rVert}      % Norm
\newcommand{\abs}[1]{\left|#1\right|}                 % Absolutbetrag
\newcommand{\paren}[1]{\left(#1\right)}               % Klammern
\newcommand{\sparen}[1]{\left\{#1\right\}}      % Mengenklammer
\newcommand{\dd}{\mathrm{d}}  % without the space
\newcommand{\Cc}{\mathcal{C}}

\newcommand{\Dc}{\mathcal{D}}
\newcommand{\Ec}{\mathcal{E}}
\newcommand{\Fc}{\mathcal{F}}

\newcommand{\Rc}{\mathcal{R}}

\newcommand{\Sc}{\mathcal{S}}

\newcommand{\WF}{\mathrm{WF}}                         % Wavefront set
\newcommand{\wf}{\mathrm{WF}}                         % Wavefront set

\newcommand{\vb}{\mathbf{b}}

\newcommand{\vr}{\mathbf{r}}

\newcommand{\partyf}[2]{\frac{\partial #2}{\partial y_{#1}}}

\newcommand{\vu}{{\mathbf{u}}}
\newcommand{\vv}{{\mathbf{v}}}

\newcommand{\ve}{{\mathbf{e}}}

\newcommand{\bpm}{\begin{pmatrix}}
\newcommand{\epm}{\end{pmatrix}}

\newcommand{\vx}{{\mathbf{x}}}

\newcommand{\vy}{{\mathbf{y}}}

\newcommand{\vs}{\mathbf{s}}
\newcommand{\vd}{\mathbf{d}}

\newcommand{\vxi}{{\boldsymbol{\xi}}}
\newcommand{\vxio}{{\boldsymbol{\xi}_0}}

\newcommand{\vsig}{{\boldsymbol{\sigma}}} %when $\sigma$ is one dimensional, 
\newcommand{\rr}{{{\mathbb R}}}

\newcommand{\rn}{{{\mathbb R}^n}}

\newcommand{\drn}{{\dot{{\mathbb R}^n}}}

\newcommand{\st}{\hskip 0.3mm : \hskip 0.3mm}

\newcommand{\be}{\begin{equation}}
\newcommand{\ee}{\end{equation}}
\newcommand{\bea}{\begin{eqnarray}}
\newcommand{\eea}{\end{eqnarray}}
\newcommand{\bean}{\begin{eqnarray*}}
\newcommand{\eean}{\end{eqnarray*}}

\newcommand{\bel}[1]{\begin{equation}\label{#1}}

\newcommand{\eel}[1]{{\label{#1}\end{equation}}}

\DeclareMathOperator{\sinc}{sinc}

\usepackage{fullpage}

\usepackage{kbordermatrix}
\renewcommand{\kbldelim}{(}% Left delimiter
\renewcommand{\kbrdelim}{)}% Right delimiter

\usepackage{datetime}
\usetikzlibrary{quotes, angles}

\newcommand\irregularcircle[2]{% radius, irregularity
  \pgfextra {\pgfmathsetmacro\len{(#1)+rand*(#2)}}
  +(0:\len pt)
  \foreach \a in {10,20,...,350}{
    \pgfextra {\pgfmathsetmacro\len{(#1)+rand*(#2)}}
    -- +(\a:\len pt)
  } -- cycle
}

\title[short]{On a cylindrical scanning modality in three-dimensional Compton scatter tomography\\{\footnotesize\ddmmyyyydate\today~\currenttime}}
\author{James W. Webber\textsuperscript{$\dagger$}}
\address[James W. Webber (corresponding author) %and someone
]{Department of Obstetrics and Gynecology, Brigham and Women's Hospital, 221 Longwood Ave, Boston, MA 02115}
\email[A1,A2]{jwebber5@bwh.harvard.edu\textsuperscript{$\dagger$} %and kelias@bwh.harvard.edu\textsuperscript{$\ddagger$}
}

\providecommand{\keywords}[1]
{
  \small	
  \textbf{\textit{Keywords---}} #1
}

\begin{document}
\maketitle
\begin{abstract}
We present injectivity and microlocal analyses of a new generalized Radon transform, $\mathcal{R}$, which has applications to a novel scanner design in three-dimensional Compton Scattering Tomography (CST), which we also introduce here. Using Fourier decomposition and Volterra equation theory, we prove that $\mathcal{R}$ is injective and show that the image solution is unique. Using microlocal analysis, we prove that $\mathcal{R}$ satisfies the Bolker condition (sometimes called the ``Bolker assumption" \cite{ABKQ2013}), and we investigate the edge detection capabilities of $\mathcal{R}$. This has important implications regarding the stability of inversion and the amplification of measurement noise. In addition, we present simulated 3-D image reconstructions from $\mathcal{R}f$ data, where $f$ is a 3-D density, with varying levels of added Gaussian noise.
 %We prove the existence of an inverse for $\mathcal{R}$ and 
 This paper provides the theoretical groundwork for 3-D CST using the proposed scanner design.
 %This work sheds light on the existence and stability of the inverse of $\mathcal{R}$, andthe amplification of measurement noise in 3-D CST.
\end{abstract}
 - cylindrical scanner, Compton tomography, generalized Radon transforms, injectivity, microlocal analysis

%\keywords{dimensionality reduction, orthogonal projections, miRNA expression analysis, cancer prediction}

\section{introduction} 
Compton scatter tomography is an imaging technique which uses Compton scattered photons to recover an electron density, which has applications in security screening, medical and cultural heritage imaging \cite{redler2018compton,rigaud20183d,cebeiro2017new,guerrero2017three,webber2020compton}.
%Numerous scanning modalities, injectivity and microlocal analyses have been investigated previously in the literature \cite{RigaudComptonSIIMS2017,rigaud20183d,rigaud20213d,cebeiro2021three,webber2021microlocal}. 
The literature presents a number of scanning modalities in two and three-dimensional CST \cite{rigaud20183d,webber2020compton,NT,me, pal, norton, me2, truong2019compton, RigaudComptonSIIMS2017, rigaud2021reconstruction, cebeiro2021three, tarpau2020analytic, godeke2022imaging, kuger2022multiple, rigaud20213d}.

In \cite[figure 5]{rigaud20183d}, the authors present a number of possible scanning modalities in 3-D CST, and derive contour reconstruction techniques using microlocal analysis and filtered backprojection ideas. The simulations focus on geometry (b) in figure 5, whereby a set of detectors are placed on a sphere and a density contained within the sphere is imaged by a single, fixed source which is located on the sphere surface. Simulated reconstructions of image phantoms are presented and Total Variation (TV) regularization is employed to help combat noise.

In \cite{cebeiro2021three}, the authors propose a new scanning geometry for 3-D CST. The geometry consists of a single source located at the origin and a set of detectors on a sphere centered at the origin, radius $R$. The difference with the geometry of \cite{rigaud20183d}, which also has a spherical array of detectors, is that in \cite{cebeiro2021three} the source is located at the center of the sphere, not on the sphere surface as in \cite{rigaud20183d}. In \cite{cebeiro2021three}, the density is also placed outside of the sphere, whereas in \cite{rigaud20183d}, the density is supported on the sphere interior. The authors in \cite{cebeiro2021three} introduce a new Radon transform, $\mathcal{R}_{\mathcal{T}}$, motivated by CST and their geometry, which integrates a compactly supported function, $f$, over apple surfaces. An apple is the exterior part of a spindle torus (see \cite[figure 4]{rigaud20183d}), which is a special type of torus that self-intersects. Using the theory of \cite{schiefeneder2017radon} on generalized Abel equations, $\mathcal{R}_{\mathcal{T}}$ is shown to be injective. Simulated reconstructions using $\mathcal{R}_{\mathcal{T}}f$ data are presented using an analytic-algebraic hybrid method.

In \cite{rigaud2021reconstruction}, the authors consider the spherical and cylindrical scanning systems first proposed in \cite[figure 5]{rigaud20183d}. In particular, the authors focus on the problem of incomplete data, which occurs when there are parts of the image edge map which are undetectable. The data incompleteness was most apparent using the cylindrical scanner, given the nature of the geometry. The authors present a modified, multiplicative Kaczmarz algorithm to address incomplete data and they test their algorithm on simulated phantoms.

In \cite{rigaud20213d}, the authors present a new model for multiple scattering in CST. The mathematical model for first and second order Compton scatter is shown to be a Fourier Integral Operator (FIO). In particular, the second scatter model is shown to be an FIO of higher order than the first scatter model. This means that the second order scatter data is more smooth than the the first order scatter. Using this idea, the authors derive contour reconstruction methods based on filtered backprojection, and validate their theory using Monte Carlo simulation.

In \cite{me2}, the authors present a 3-D CST scanning modality. The geometry consists of a single source and detector pair which are rotated opposite one another on the surface of the unit sphere. The scanned object is placed inside the sphere. The authors introduce a new Radon transform, $\mathcal{S}$, which integrates a smooth function of compact support over spindle surfaces, which are sometimes called ``lemons" elsewhere in the literature \cite{rigaud20183d}. A ``lemon" or ``spindle" is the surface of rotation of a circular arc. Equivalently, a lemon is the interior part of a spindle torus (see \cite[figure 4]{rigaud20183d}). Using spherical harmonic expansion ideas, and classical theory on Volterra equations, the authors show that $\mathcal{S}$ is injective when $f$ is supported within the upper unit hemisphere, and they derive an inversion method based on Neumann series. Reconstructions of simulated phantoms are presented using Tikhonov regularization.

\begin{figure}[!h]
\centering
\begin{tikzpicture}[scale=0.6]
\draw[fill=green,rounded corners=1mm] (0-1,0-1) \irregularcircle{1.9cm}{1mm};
\draw [->,line width=1pt] (-5,0)--(5,0)node[right] {$x$};
\draw [->,line width=1pt] (0,-5)--(0,5)node[right] {$y$};
\node at (-0.9-1,1.2) {$f$};
\draw [red] (0,0) circle (4.5);
\draw (4.5,0) circle (2);
\coordinate (x) at (3,1.3229);
\coordinate (O) at (0,0);
\coordinate (c) at (4.5,0);
%\draw (O)--(x);
%\draw (c)--(x);
%\draw pic[draw=orange, <->,"$\phi$", angle eccentricity=1.5] {angle = c--O--x};
%\draw pic[draw=orange, <->,"$\theta$", angle eccentricity=1.5] {angle = x--c--O};
%\node at (1.5,0.9) {$r$};
%\node at (3.9,0.9) {$t$};
%\draw [<->] (0,-0.1)--(4.5,-0.1);
%\node at (-2.25,-0.3) {1};
\node at (6,2) {$\mathcal{L}$};
\draw [fill=black] (-3,-4.1+1) rectangle (3,-3.9+1);
\draw [->] (4.5,-4.1)node[right]{conveyor}--(3,-3.1);
\end{tikzpicture}
\begin{tikzpicture}[scale=0.6]
\draw[fill=green,rounded corners=1mm] (0-1,2.6458) \irregularcircle{1.5cm}{1mm};
\draw [<-,line width=1pt] (-5,-1.5)node[left] {$-x$}--(4.5,-1.5);
\draw [<->] (-0.2,-1.5)--(-0.2,0);
\node at (-0.55,-0.75) {$z_0$};
\draw [->,line width=1pt] (0,-5)--(0,5)node[right] {$z$};
%\draw (4.5,0)--(6.8,0);
%\node at (4.5+1.15,-0.3) {$R$};
%\draw (4.5,3.9837)--(6.8,0);
%\node at (4.5+1.5,2.1) {$p$};
\node at (-0.9-1,2.6458+2.1-0.4) {$f$};
\coordinate (x) at (3,2.6458);
\coordinate (O) at (0,0);
\coordinate (c) at (4.5,0);
%\draw (O)--(x);
%\draw (c)--(x);
%\draw (x)--(4.5,2.6458);
%\node at (3.9,3) {$t$};
%\node at (4.8,1.25) {$z$};
\node at (-2.25,-0.3-1.5) {1};
\draw [->] (-2,-1.8)--(0,-1.8);
\draw [->] (-2.5,-1.8)--(-4.5,-1.8);
\draw [blue] (4.5,-5)--(4.5,0);
\draw [red] (4.5,0)--(4.5,5);
\draw [blue] (-4.5,-5)--(-4.5,0);
\draw [red] (-4.5,0)--(-4.5,5);
%\draw (6.8,0) circle (4.6);
%\node at (3.4,-0.3) {$h$};
%\node at (5.8,-2.1) {$\sqrt{p^2-R^2}$};
\draw [dashed] (-4.5,2.6458)--(4.5,2.6458);

\draw [thick,domain=120:240] plot ({6.8+4.6*cos(\x)}, {4.6*sin(\x)});
\draw [thick,domain=-60:60] plot ({-6.8+2*4.5+4.6*cos(\x)}, {4.6*sin(\x)});
\node at (4.75,-4.25) {$\vs$};
\node at (4.8,4.3) {$\vd$};
\node at (6.5,2.8) {$\mathcal{L}$};

%\draw [snake it][->] (4.5,-4)--(2.3,1);
%\draw [snake it][->] (2.3,1)--(4.5,4);
%\draw [->] (2.3,1)--(1.8,2);
%\coordinate (D) at (4.5,4);
%\coordinate (w) at (2.3,1);
%\coordinate (a) at (1.8,2);
%\draw pic[draw=orange, <->,"$\omega$", angle eccentricity=1.5] {angle = D--w--a};
\draw [snake it][->] (4.5,-4)--(2.2997,0.9526);
\draw [snake it][->] (2.2997,0.9526)--(4.5,4);
\draw [->] (2.2997,0.9526)--(1.8125,2.0493);
\coordinate (D) at (4.5,4);
\coordinate (w) at (2.2997,0.9526);
\coordinate (a) at (1.8125,2.0493);
\draw pic[draw=orange, <->,"$\omega$", angle eccentricity=1.5] {angle = D--w--a};
\node at (4,-2) {$E$};
\node at (3.68,2) {$E'$};
\end{tikzpicture}

\caption{$(x,y)$ and $(x,z)$ plane cross sections of the proposed cylindrical scanning geometry. The cylinder has unit radius. The $(x,y)$ plane of the left-hand figure is highlighted as a dashed line in the $(x,z)$ plane. Cross-sections of a lemon of integration, $\mathcal{L}$, are labeled. The sources ($\vs$) are located on the bottom half of the cylinder, highlighted in blue, and the detectors ($\vd$), highlighted in red, are located on the upper half. The scanning target, $f$, displayed as a green, irregular disc, passes through the cylinder on a conveyor belt in the $z$ direction. The variable $z_0 \in \mathbb{R}$ is the $z$ component of the center of $\mathcal{L}$.}
\label{fig2}
\end{figure}
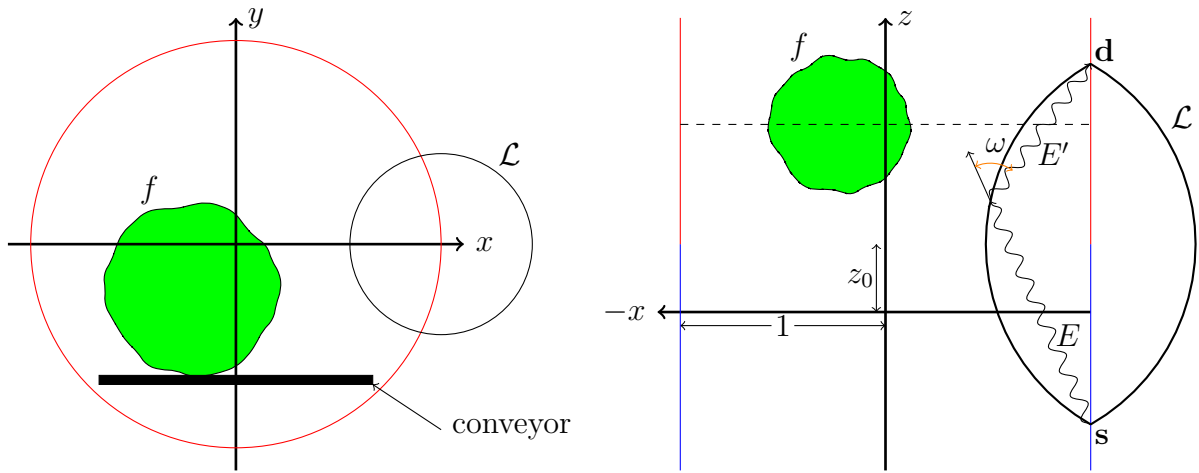

In this paper, we introduce a new scanning modality in 3-D CST, whereby monochromatic (e.g., gamma ray) sources and energy-sensitive detectors on a cylindrical surface scan a density passing through the cylinder on a conveyor. See figure \ref{fig2}, where we have illustrated $(x,y)$ and $(x,z)$ plane cross-sections of the proposed scanner geometry. The incoming photons, which are emitted from $\vs$ with energy $E$, Compton scatter from charged particles (usually electrons) with energy $E'$, and are measured by the detector $\vd$; meanwhile, the electron charge density, $f$ (represented by a real-valued function), passes through the cylinder in the $z$ direction on a conveyor belt. The scattered energy, $E'$, is given by the equation
\begin{equation}
\label{equ1}
E'=\frac{E}{1+(E/E_0)(1-\cos\omega)},
\end{equation}
where $E$ is 
the initial energy, $\omega$ is the scattering angle and $E_0\approx 511\text{keV}$ denotes
the electron rest energy. If the source is monochromatic (i.e., $E$ is
fixed) and we can measure the scattered energy, $E'$, i.e., the detectors are energy-sensitive, then the
scattering angle, $\omega$, of the interaction is fixed and determined by equation
\eqref{equ1}. This implies that the surface of Compton scatterers is the surface of rotation of a circular arc, which we denote as a lemon. An example 2-D cross section of a lemon, $\mathcal{L}$, is shown in figure \ref{fig2}. The axis of rotation of $\mathcal{L}$ is parallel to the $z$ axis and embedded within the cylinder surface. We model the Compton scattered intensity as integrals of $f$
over lemons. See, e.g., \cite{rigaud20183d} for other work which models the Compton intensity in this way. The proposed geometry illustrated in figure \ref{fig2} has similarities to that of \cite[figure 5, subfigure (c)]{rigaud20183d}, although we consider multiple sources on a cylindrical surface. In \cite[figure 5, subfigure (c)]{rigaud20183d} only one, fixed source is used for imaging.

The geometry and physical modeling leads us to a new Radon transform, $\mathcal{R}$, which integrates $f$ over lemon surfaces. Using Fourier decomposition and Volterra equation theory \cite{tricomi1985integral}, we prove that $\mathcal{R}$ is injective, which implies $f$ can be uniquely recovered using Compton scatter data. The proof uses similar ideas to \cite[Theorem 4.3]{webber2022ellipsoidal}, where the authors prove injectivity for a spheroid Radon transform. Using the theory of linear FIO, we prove that $\mathcal{R}$ satisfies the Bolker condition \cite{ABKQ2013}, which gives insight into the reconstruction artifacts. Using microlocal analysis, we investigate the edge detection capabilities of $\mathcal{R}$ and discuss how this relates to image edge reconstruction. In addition, we present simulated 3-D image reconstructions from $\mathcal{R}f$ data with varying levels of added Gaussian noise. The results presented here provide a novel framework for CST, and lay the theoretical foundation for 3-D density reconstruction using the proposed scanner design.

The remainder of this paper is organized as follows. In section \ref{sect:defns}, we review some definitions from microlocal analysis that will be used in our theorems. In section \ref{section_spindle}, we introduce a new lemon Radon transform, $\mathcal{R}$, and prove injectivity. In section \ref{microlocal_section}, we analyze the stability of $\mathcal{R}$ using microlocal analysis. We show that $\mathcal{R}$ satisfies the Bolker condition and we investigate the edge detection capabilities of $\mathcal{R}$. In section \ref{results}, we validate our microlocal theory, and present image reconstructions of simulated phantoms from $\mathcal{R}f$ data with varying levels of added Gaussian noise.

\section{Definitions from microlocal analysis}\label{sect:defns} 

In this section, we review some theory from microlocal analysis which will be used in our theorems. We first provide some
notation and definitions.  Let $X$ and $Y$ be open subsets of
{$\mathbb{R}^{n_X}$ and $\mathbb{R}^{n_Y}$, respectively.}  Let $\Dc(X)$ be the space of smooth functions compactly
supported on $X$ with the standard topology and let $\mathcal{D}'(X)$
denote its dual space, the vector space of distributions on $X$.  Let
$\Ec(X)$ be the space of all smooth functions on $X$ with the standard
topology and let $\mathcal{E}'(X)$ denote its dual space, the vector
space of distributions with compact support contained in $X$. Finally,
let $\Sc(\rn)$ be the space of Schwartz functions, that are rapidly
decreasing at $\infty$ along with all derivatives. See \cite{Rudin:FA}
for more information. 

We now list some notation conventions that will be used throughout this paper:
\begin{enumerate}
\item For a function $f$ in the Schwartz space $\Sc(\mathbb{R}^{n_X})$ or in
$L^2(\rn)$, we use $\mathcal{F}f$ and $\mathcal{F}^{-1}f$ to denote
the Fourier transform and inverse Fourier transform of $f$,
respectively (see \cite[Definition 7.1.1]{hormanderI}). $\mathcal{F}f$ and $\mathcal{F}^{-1}f$ are defined in terms of angular frequency.
%Note that
%$$\Fc\inv \Fc f(\vx)= \frac{1}{(2\pi)^{n_X}}\int_{\vy\in\mathbb{R}^{n_X}}\int_{\vz\in
%\mathbb{R}^{n_X}} \exp((\vx-\vz)\cdot \vy)\,
%f(\vz)\d \vz\d \vy.$$

\item We use the standard multi-index notation: if
$\alpha=(\alpha_1,\alpha_2,\dots,\alpha_n)\in \sparen{0,1,2,\dots}^{n_X}$
is a multi-index and $f$ is a function on $\mathbb{R}^{n_X}$, then
\[\partial^\alpha f=\paren{\frac{\partial}{\partial
x_1}}^{\alpha_1}\paren{\frac{\partial}{\partial
x_2}}^{\alpha_2}\cdots\paren{\frac{\partial}{\partial x_{n_X}}}^{\alpha_{n_X}}
f.\] If $f$ is a function of $(\vy,\vx,\vsig)$ then $\partial^\alpha_\vy
f$ and $\partial^\alpha_\vsig f$ are defined similarly.

\item We identify the cotangent
spaces of Euclidean spaces with the underlying Euclidean spaces. For example, the cotangent space, 
$T^*(X)$, of $X$ is identified with $X\times \mathbb{R}^{n_X}$. If $\Phi$ is a function of $(\vy,\vx,\vsig)\in Y\times X\times \rr^N$,
then we define $\dd_{\vy} \Phi = \paren{\partyf{1}{\Phi},
\partyf{2}{\Phi}, \cdots, \partyf{{n_X}}{\Phi} }$, and $\dd_\vx\Phi$ and $
\dd_{\vsig} \Phi $ are defined similarly. Identifying the cotangent space with the Euclidean space as mentioned above, we let $\dd\Phi =
\paren{\dd_{\vy} \Phi, \dd_{\vx} \Phi,\dd_{\vsig} \Phi}$.

\item For $\Omega\subset \rr^m$, we define $\dot{\Omega}
= \Omega\smo$.

\end{enumerate}

\noindent The singularities of a function and the directions in which they occur
are described by the wavefront set \cite[page
16]{duistermaat1996fourier}, which we define below.
\begin{definition}
\label{WF} Let $X$ be an open subset of $\rn$ and let $f$ be a
distribution in $\mathcal{D}'(X)$.  Let $(\vx_0,\vxi_0)\in X\times
\drn$.  Then $f$ is \emph{smooth at $\vx_0$ in direction $\vxio$} if
there exists a neighborhood $U$ of $\vx_0$ and $V$ of $\vxi_0$ such
that for every $\Phi\in \Dc(U)$ and $N\in\mathbb{R}$ there exists a
constant $C_N$ such that for all $\vxi\in V$,
\begin{equation}
\left|\Fc(\Phi f)(\lambda\vxi)\right|\leq C_N(1+\abs{\lambda})^{-N}.
\end{equation}
The pair $(\vx_0,\vxio)$ is in the \emph{wavefront set,} $\wf(f)$, if
$f$ is not smooth at $\vx_0$ in direction $\vxio$.
\end{definition}
 This definition follows the intuitive idea that the elements of
$\WF(f)$ are the point-normal vector pairs at
which $f$ has singularities.  For example, if $f$ is the
characteristic function on the unit ball, $B_n = \{\vx\in\mathbb{R}^n : |\vx|\leq 1\}$, in $\mathbb{R}^n$, then its
wavefront set is $\WF(f)=\{(\vx,t\vx): \vx\in S^{n-1}, t\neq 0\}$, i.e., the
set of points on $S^{n-1}$ (i.e., the boundary of $B$) paired with the corresponding normal vectors
to $S^{n-1}$.

%\begin{equation}

%\end{equation}
%That is, 

The wavefront set of a distribution on $X$ is normally defined as a
subset the cotangent bundle $T^*(X)$ so it is invariant under
diffeomorphisms, but we do not need this invariance, so we will
continue to identify $T^*(X) = X \times \rn$ and consider $\WF(f)$ as
a subset of $X\times \drn$.

%Let $X$ and $Y$ be open subsets of $\rn$, $m \in\mathbb{R}$.

 \begin{definition}[{\cite[Definition 7.8.1]{hormanderI}}] \label{ellip}We define
 $S^m(Y \times X, \mathbb{R}^N)$ to be the
set of $a\in \Ec(Y\times X\times \mathbb{R}^N)$ such that for every
compact set $K\subset Y\times X$ and all multi--indices $\alpha,
\beta, \gamma$ the bound
\[
\left|\partial^{\gamma}_{\vy}\partial^{\beta}_{\vx}\partial^{\alpha}_{\vsig}a(\vy,\vx,\vsig)\right|\leq
C_{K,\alpha,\beta,\gamma}(1+\norm{\vsig})^{m-|\alpha|},\ \ \ (\vy,\vx)\in K,\
\vsig\in\mathbb{R}^N,
\]
holds for some constant $C_{K,\alpha,\beta,\gamma}>0$. 

 The elements of $S^m$ are called \emph{symbols} of order $m$.  Note
that these symbols are sometimes denoted $S^m_{1,0}$.  The symbol
$a\in S^m(Y \times X,\rr^N)$ is \emph{elliptic} if for each compact set
$K\subset Y\times X$, there is a $C_K>0$ and $M>0$ such that
\bel{def:elliptic} \abs{a(\vy,\vx,\vsig)}\geq C_K(1+\norm{\vsig})^m,\
\ \ (\vy,\vx)\in K,\ \norm{\vsig}\geq M.
\ee 
\end{definition}

\begin{definition}[{\cite[Definition
        21.2.15]{hormanderIII}}] \label{phasedef}
A function $\Phi=\Phi(\vy,\vx,\vsig)\in
\Ec(Y\times X\times\dot{\mathbb{R}^N})$ is a \emph{phase
function} if $\Phi(\vy,\vx,\lambda\vsig)=\lambda\Phi(\vy,\vx,\vsig)$, $\forall
\lambda>0$ and $\mathrm{d}\Phi$ is nowhere zero. The
\emph{critical set of $\Phi$} is
\[\Sigma_\Phi=\{(\vy,\vx,\vsig)\in Y\times X\times\dot{\mathbb{R}^N}
: \dd_{\vsig}\Phi=0\}.\] 
 A phase function is
\emph{clean} if the critical set $\Sigma_\Phi = \{ (\vy,\vx,\vsig) \ : \
\mathrm{d}_\vsig \Phi(\vy,\vx,\vsig) = 0 \}$ is a smooth manifold {with tangent space defined {by} the kernel of $\mathrm{d}\,(\mathrm{d}_\sigma\Phi)$ on $\Sigma_\Phi$. Here, the derivative $\mathrm{d}$ is applied component-wise to the vector-valued function $\mathrm{d}_\sigma\Phi$. So, $\mathrm{d}\,(\mathrm{d}_\sigma\Phi)$ is treated as a Jacobian matrix of dimensions $N\times (2n+N)$.}
\end{definition}
\noindent By the {Constant Rank Theorem} the requirement for a phase
function to be clean is satisfied if
$\mathrm{d}\paren{\mathrm{d}_\vsig
\Phi}$ has constant rank.

\begin{definition}[{\cite[Definition 21.2.15]{hormanderIII} and
      \cite[section 25.2]{hormander}}]\label{def:canon} Let $X$ and
$Y$ be open subsets of $\rn$. Let $\Phi\in \Ec\paren{Y \times X \times
{\rr}^N}$ be a clean phase function.  In addition, we assume that
$\Phi$ is \emph{nondegenerate} in the following sense:
\[\text{$\dd_{\vy}\Phi$ and $\dd_{\vx}\Phi$ are never zero on
$\Sigma_{\Phi}$.}\]
  The
\emph{canonical relation parametrized by $\Phi$} is defined as
\begin{equation}\label{def:Cgenl} \begin{aligned} \Cc=&\sparen{
\paren{\paren{\vy,\dd_{\vy}\Phi(\vy,\vx,\vsig)};\paren{\vx,-\dd_{\vx}\Phi(\vy,\vx,\vsig)}}:(\vy,\vx,\vsig)\in
\Sigma_{\Phi}},
% &\hspace{1.5cm} \vs\in \rr^N\smo,   
\end{aligned}
\end{equation}
\end{definition}

\begin{definition}\label{FIOdef}
Let $X$ and $Y$ be open subsets of {$\mathbb{R}^{n_X}$ and $\mathbb{R}^{n_Y}$, respectively.} {Let an operator $A :
\Dc(X)\to \mathcal{D}'(Y)$ be defined by the distribution kernel
$K_A\in \mathcal{D}'(Y\times X)$, in the sense that
$Af(\vy)=\int_{X}K_A(\vy,\vx)f(\vx)\mathrm{d}\vx$. Then we call $K_A$
the \emph{Schwartz kernel} of $A$}. A \emph{Fourier
integral operator (FIO)} of order $m + N/2 - (n_X+n_Y)/4$ is an operator
$A:\Dc(X)\to \mathcal{D}'(Y)$ with Schwartz kernel given by an
oscillatory integral of the form
\begin{equation} \label{oscint}
K_A(\vy,\vx)=\int_{\mathbb{R}^N}
e^{i\Phi(\vy,\vx,\vsig)}a(\vy,\vx,\vsig) \mathrm{d}\vsig,
\end{equation}
where $\Phi$ is a clean nondegenerate phase function and $a$ is a
symbol in $S^m(Y \times X , \mathbb{R}^N)$. The \emph{canonical
relation of $A$} is the canonical relation of $\Phi$ defined in
\eqref{def:Cgenl}. $A$ is called an \emph{elliptic} FIO if its symbol is elliptic.

An FIO is called a \emph{pseudodifferential operator} if its canonical relation $\Cc$ is contained in the diagonal, i.e.,
$\Cc \subset \Delta := \{ (\vx,\vxi;\vx,\vxi)\}$.
\end{definition}
%\vspace{0.4cm}
\noindent Let $X$ and $Y$ be
sets and let $\Omega_1\subset X$ and $\Omega_2\subset Y\times X$. The composition $\Omega_2\circ \Omega_1$ and transpose $\Omega_2^t$ of $\Omega_2$ are defined
\[\begin{aligned}\Omega_2\circ \Omega_1 &= \sparen{\vy\in Y\st \exists \vx\in \Omega_1,\
(\vy,\vx)\in \Omega_2}\\
\Omega_2^t &= \sparen{(\vx,\vy)\st (\vy,\vx)\in \Omega_2}.\end{aligned}\]
We now state the H\"ormander-Sato Lemma \cite[Theorem 8.2.13]{hormanderI}, which explains the relationship between the
wavefront set of distributions and their images under FIO.

\begin{theorem}[H\"ormander-Sato Lemma]\label{thm:HS} Let $f\in \Ec'(X)$ and
let ${A}:\Ec'(X)\to \Dc'(Y)$ be an FIO with canonical relation $\Cc$.
Then, $\wf({A}f)\subset \Cc\circ \wf(f)$.\end{theorem}

Let $A$ be an FIO with adjoint $A^*$. Then if $\Cc$ is the canonical relation of $A$, the canonical relation of $A^*$ is $\Cc^t$. Many imaging techniques are based on application of the adjoint operator $A^*$ and so to understand artifacts we consider $A^* A$ (or, if
$A$ does not map to $\Ec'(Y)$, then $A^* \psi A$ for an
appropriate cutoff $\psi$). Because of Theorem \ref{thm:HS},
\begin{equation}
\label{compo}
\wf(A^* \psi A f) \subset \Cc^t \circ \Cc \circ \wf(f).
\end{equation}
The next two definitions provide tools to analyze the composition in equation \eqref{compo}, which we will apply later in section \ref{microlocal_section}.

\begin{definition}
\label{defproj} Let $\Cc\subset T^*(Y\times X)$ be the canonical
relation associated to the FIO ${A}:\mathcal{E}'(X)\to
\mathcal{D}'(Y)$. We let $\Pi_L$ and $\Pi_R$ denote the natural left-
and right-projections of $\Cc$, projecting onto the appropriate
coordinates: $\Pi_L:\Cc\to T^*(Y)$ and $\Pi_R : \Cc\to T^*(X)$.
\end{definition}

Because $\Phi$ is nondegenerate, the projections do not map to the
zero section.  
% 
% We have the following result from \cite{hormander}.
% \begin{proposition}
% \label{prop1}
% Let $\dim(X)=\dim(Y)$. Then at any point in $\Cc$:
% \begin{enumerate}[(i)]
% \item if one of $\Pi_L$ or $\Pi_R$ is a local diffeomorphism, then the
% other map is a local diffeomorphism (so $\Cc$ is a local canonical
% graph); 
% 
% \item if one of the projections $\Pi_R$ or $\Pi_L$ is singular at a
% point in $\Cc$, then so is the other. The type of the singularity may
% be different but both projections drop rank on the same set
% \begin{equation}
% \Sigma=\{(\vy,\eta; \vx,\vsig)\in \Cc :
% \det(\mathrm{d}\Pi_L)=0\}=\{(\vy,\eta; \vx,\vsig)\in \Cc : \det
% (\mathrm{d}\Pi_R)=0\}.
% \end{equation}
% \end{enumerate}
% \end{proposition}
If $A$ satisfies our next definition, then $A^* A$ (or $A^* \psi A$) is a pseudodifferential operator
\cite{GS1977, quinto}.

\begin{definition}\label{def:bolker} Let
${A}:\Ec'(X)\to \Dc'(Y)$ be a FIO with canonical relation $\Cc$ then
{$A$} (or $\Cc$) satisfies the \emph{Bolker Condition} if
the natural projection $\Pi_L:\Cc\to T^*(Y)$ is an embedding
(injective immersion).\end{definition}

\section{A Lemon Radon transform on a cylinder}
\label{section_spindle}
In this section, we introduce a new Radon transform, $\mathcal{R}$, which integrates a square integrable function of compact support over lemon surfaces. A ``lemon" is the surface of rotation of a circular arc. The lemon surfaces we consider have axes of rotation which are embedded in the cylindrical scanning surface of figure \ref{fig2}. In our first main theorem in this section, we prove that $\mathcal{R}$ is injective. 

We use the standard cylindrical coordinate system $(x,y,z)=(r\cos\theta,r\sin\theta,z)$, where $r\in [0,\infty)$, $\theta \in [0,2\pi]$, and $z\in\mathbb{R}$. Refer to figure \ref{fig1}, which illustrates coordinates on the lemon surfaces of integration in the scanning geometry of figure \ref{fig2}.
\begin{figure}[!h]
\centering
\begin{tikzpicture}[scale=0.5]
\draw [->,line width=1pt] (-5,0)--(5,0)node[right] {$x$};
\draw [->,line width=1pt] (0,-5)--(0,5)node[right] {$y$};
\draw[green,rounded corners=1mm] (0-1,0+1) \irregularcircle{0.75cm}{1mm};
\node at (-0.9-1,1+1) {$f$};
\draw (0,0) circle (4.5);
\draw (4.5,0) circle (2);
\coordinate (x) at (3,1.3229);
\coordinate (O) at (0,0);
\coordinate (c) at (4.5,0);
\draw (O)--(x);
\draw (c)--(x);
\draw pic[draw=orange, <->,"$\theta$", angle eccentricity=1.75] {angle = c--O--x};
\draw pic[draw=orange, <->,"$\phi$", angle eccentricity=1.5] {angle = x--c--O};
\node at (1.5,1) {$r$};
\node at (3.9,1) {$t$};
%\draw [<->] (0,-0.1)--(4.5,-0.1);
\node at (2.25,-0.4) {1};
\end{tikzpicture}
\begin{tikzpicture}[scale=0.5]
\draw [<-,line width=1pt] (-4.5,0)node[left] {$-x$}--(4.5,0);
\draw [->,line width=1pt] (0,-5)--(0,5)node[right] {$z$};
\draw[green,rounded corners=1mm] (0-1,2.6458) \irregularcircle{0.75cm}{1mm};
\draw (4.5,0)--(6.8,0);
\node at (4.5+1.15,-0.475) {$R$};
\draw (4.5,3.9837)--(6.8,0);
\node at (4.5+1.5,2.2) {$p$};
\node at (-0.9-1,2.6458+1) {$f$};
\coordinate (x) at (3,2.6458);
\coordinate (O) at (0,0);
\coordinate (c) at (4.5,0);
%\draw (O)--(x);
%\draw (c)--(x);
\draw (x)--(4.5,2.6458);
\node at (3.9,3) {$t$};
%\node at (4.8,1.25) {$z$};
%\node at (-2.25,-0.3) {1};
\draw (4.5,-5)--(4.5,5);
\draw (-4.5,-5)--(-4.5,5);
\draw (6.8,0) circle (4.6);
\node at (3.4,-0.5) {$h$};
\node at (6.3,-2.2) {$\sqrt{p^2-R^2}$};
\draw [dashed] (-4.5,2.6458)--(4.5,2.6458);
\end{tikzpicture}
\caption{Cylindrical scanning geometry with coordinates.}
\label{fig1}
\end{figure}
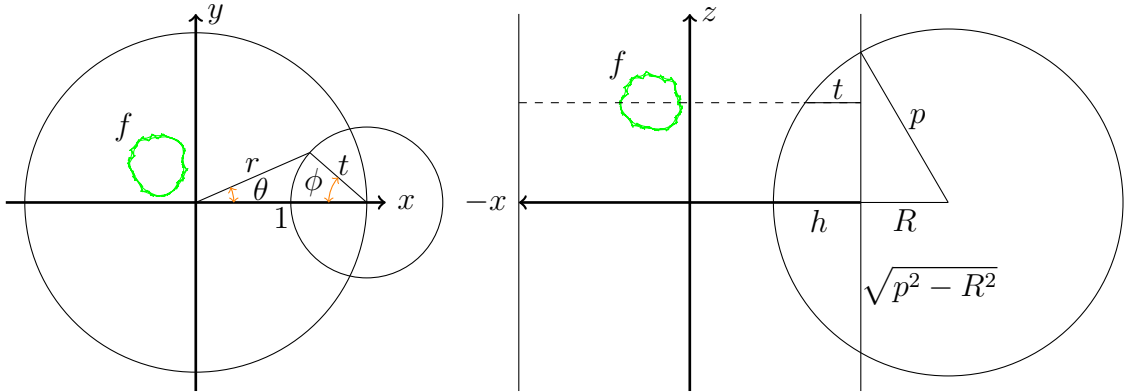
Then, we have the coordinate transformations
$$t=\sqrt{p^2-z^2} -R,\ \ \ h=p-R,$$
$$r=\sqrt{t^2+1-2t\cos\phi},\ \ \ \frac{t}{\sin\theta}=\frac{r}{\sin\phi}.$$
%Although the injectivity results below hold for all $s>0$, e are interested in $s \in (0,1]$, so that the spheroid foci lie on the cylinder of sources and receivers. 
Let $\mathrm{d}s=\sqrt{1+\paren{\frac{\mathrm{d}t}{\mathrm{d}z}}^2}\mathrm{d}z$ be the arc measure on the circular arc with height $h$ in the right-hand figure of figure \ref{fig2}. Then the surface element on the circular arc of revolution (i.e., a lemon) is
\begin{equation}
\begin{split}
\mathrm{d}A&=t\mathrm{d}s\mathrm{d}\phi\\
&=\frac{tp}{\sqrt{p^2-z^2}}\mathrm{d}\phi\mathrm{d}z\\
\end{split}
\end{equation}
Let $L^2_c(\Omega)$ denote the set of square integrable functions with compact support on $\Omega \subset \mathbb{R}^3$. Let $f\in L^2_c(C_{\epsilon})$, where 
$$C_{\epsilon}=\{(x,y,z)\in\mathbb{R}^3 : \sqrt{x^2+y^2} \leq 1-\epsilon\},$$
for some small offset $0<\epsilon<1$. We define the Radon transform
\begin{equation}
\label{Rf}
\mathcal{R}f(p,R,\theta_0,z_0)=p\int_{-\sqrt{p^2-R^2}}^{\sqrt{p^2-R^2}}\frac{t}{\sqrt{p^2-z^2}}\int_{-\pi}^{\pi}f\paren{r,\sin^{-1}\paren{\frac{t}{r}\sin\phi}+\theta_0,z+z_0}\mathrm{d}\phi\mathrm{d}z,
\end{equation}
which defines the integrals of $f$ over lemons, with central axis $\{(\cos\theta_0,\sin\theta_0,z) : z\in\mathbb{R}\}$, where $(\cos\theta_0,\sin\theta_0,z_0)$ is the center of the lemon, and $p$ and $R$ are the radii of the lemon, as illustrated in figure \ref{fig1}. The picture shown in figure \ref{fig1} corresponds to a lemon with center $(1,0,0)$. To relate the variables $p,R,\theta_0,z_0$ to physical quantities and figure \ref{fig2}, the source position, $\vs$, detector position, $\vd$, and scattered energy, $E'$, determine $p$, $R$, and $\theta_0$. For example, as $\vs$ and $\vd$ move further apart and $E'$ stays fixed, $p$ increases. The position of $f$ on the conveyor determines $z_0$, which induces translation on $f$ in the $z$ direction.

We define the limited data transform
\begin{equation}
\mathcal{R}_Lf(h,\theta_0,z_0) =  \mathcal{R}f\paren{\frac{\alpha^2+h^2}{2h},\frac{\alpha^2-h^2}{2h},\theta_0,z_0},
\end{equation}
where $h\in (0,\alpha]$ is the height of the lemon as in figure \ref{fig1}. In this case, the distance between the source and detector is fixed at the value $2\alpha$, and $\sqrt{p^2-R^2} = \alpha$. To measure $\Rc f$, sources and detectors would need to be placed at every point on the lower and upper half cylinder, respectively, as depicted in figure \ref{fig2}. In this case, the set of source and detector positions is two-dimensional. To measure $\Rc_L f$, the sources and detectors are placed on two circles (rings) a fixed distance $2\alpha$ apart. The source and detector rings are the intersections of the planes $\{z=-\alpha\}$ and $\{z=\alpha\}$, respectively, with the boundary of $C_0$, and the set of source and detector positions is one-dimensional. In the following subsections, we explore the injectivity and stability properties of $\Rc$ and $\Rc_L$

\subsection{Injectivity and inversion method}
In this subsection, we prove injectivity of $\mathcal{R}_L$ and $\Rc$, and provide an inversion method using the theory of Volterra equations \cite{tricomi1985integral} and Neumann series. We now have our first main theorem.

\begin{theorem}
$\Rc$ and $\Rc_L$ are injective on domain $L^2_c(C_{\epsilon})$, for any fixed $\epsilon \in (0,1)$.

\begin{proof}
%Let us define
%$$p = p(h) = \frac{\alpha^2+h^2}{2h}, \ \ \ R = R(h) = p-h = \frac{\alpha^2-h^2}{2h}.$$
Taking the Fourier transform in $z_0$ on both sides of \eqref{Rf} yields
\begin{equation}
\label{equ_p_1}
\widehat{\mathcal{R}f}(p,R,\theta_0,\eta)=2p\int_{0}^{\sqrt{p^2-R^2}}\cos(\eta z)\frac{t}{\sqrt{p^2-z^2}}\int_{-\pi}^{\pi}\hat{f}\paren{r,\sin^{-1}\paren{\frac{t}{r}\sin\phi}+\theta_0,\eta}\mathrm{d}\phi\mathrm{d}z,
\end{equation}
where $\eta$ is dual to $z$. Taking the Fourier components in $\theta_0$ on both sides of \eqref{equ_p_1} yields
\begin{equation}
\label{equ_p_2}
\widehat{\mathcal{R}f}_n(p,R,\eta)=4\int_{0}^{\sqrt{p^2-R^2}}\cos(\eta z)\frac{tp}{\sqrt{p^2-z^2}}\int_{0}^{\pi}\cos(n\theta)\hat{f}_n\paren{r,\eta}\mathrm{d}\phi\mathrm{d}z,
\end{equation}
for $n\in \mathbb{Z}$, where 
$$\hat{f}_n\paren{r,\eta}=\frac{1}{2\pi}\int_{0}^{2\pi}\hat{f}(r,\theta_0,\eta)e^{-i n \theta_0}\mathrm{d}\theta_0,$$
and
$$\widehat{\mathcal{R}f}_n(p,R,\eta)=\frac{1}{2\pi}\int_{0}^{2\pi}\widehat{\mathcal{R}f}(p,R,\theta_0,\eta)e^{-i n \theta_0}\mathrm{d}\theta_0.$$
Substituting $t=\sqrt{p^2-z^2} -R$ in the $z$ integral of \eqref{equ_p_2}yields
\begin{equation}
\label{equ_p_3}
\widehat{\mathcal{R}f}_n(p,R,\eta)=4\int_{0}^{h}\frac{t(t+R)}{\sqrt{p^2-(t+R)^2}}\cos\paren{\eta\sqrt{p^2-(t+R)^2}}\int_{0}^{\pi}\cos(n\theta)\hat{f}_n\paren{r,\eta}\mathrm{d}\phi\mathrm{d}t,
\end{equation}
where $r=\sqrt{t^2+1-2t\cos\phi}$, and $\sin\theta=\frac{t}{r}\sin\phi$. 

Let us now substitute $r=\sqrt{t^2+1-2t\cos\phi}$ in the $\phi$ integral of \eqref{equ_p_3}. We have
$$\frac{\mathrm{d}r}{\mathrm{d}\phi}=\frac{t}{r}\sin\phi=\sin\theta,\ \ \  \cos\theta=\frac{r^2+1-t^2}{2r}.$$
Then
\begin{equation}
\begin{split}
&\widehat{\mathcal{R}f}_n(p,R,\eta)=\\
&4\int_{0}^{h}\frac{t(t+R)}{\sqrt{p^2-(t+R)^2}}\cos\paren{\eta\sqrt{p^2-(t+R)^2}}\int_{1-t}^{1-\epsilon}\frac{T_{|n|}\paren{\frac{r^2+1-t^2}{2r}}}{\sqrt{1-\paren{\frac{r^2+1-t^2}{2r}}^2}}\hat{f}_n\paren{r,\eta}\mathrm{d}r\mathrm{d}t\\
&=4\int_{1-h}^{1-\epsilon}\int_{1-r}^{h}\frac{t(t+R)\cos\paren{\eta\sqrt{p^2-(t+R)^2}}T_{|n|}\paren{\frac{r^2+1-t^2}{2r}}}{\sqrt{p^2-(t+R)^2}\sqrt{1-\paren{\frac{r^2+1-t^2}{2r}}^2}}\hat{f}_n\paren{r,\eta}\mathrm{d}t\mathrm{d}r,
\end{split}
\end{equation}
where $T_{|n|}$ is a Chebyshev polynomial degree $|n|$. Substituting $r=1-u$ yields
\begin{equation}
\label{volt}
\begin{split}
&\widehat{\mathcal{R}f}_n(p,R,\eta)=4\int_{\epsilon}^{h}K_n(\eta ; p,R,u)\tilde{\hat{f}}_n\paren{u,\eta}\mathrm{d}u,
\end{split}
\end{equation}
a Volterra equation  of the first kind, where $\tilde{\hat{f}}_n\paren{u,\eta}=\hat{f}_n\paren{1-u,\eta}$, and
\begin{equation}
\label{B.9}
\begin{split}
&K_n(\eta ; p,R,u)=\\
&2(1-u)\int_{u}^{h}\frac{t(t+R)\cos\paren{\eta\sqrt{p^2-(t+R)^2}}T_{|n|}\paren{\frac{(1-u)^2+1-t^2}{2(1-u)}}}{\sqrt{p^2-(t+R)^2}\sqrt{t^2-u^2}\sqrt{2+t-u}\sqrt{2-t-u} }\mathrm{d}t\\
&=2(1-u)\int_{0}^{1}\frac{t(t+R)\cos\paren{\eta\sqrt{p^2-(t+R)^2}}T_{|n|}\paren{\frac{(1-u)^2+1-t^2}{2(1-u)}}}{\sqrt{v}\sqrt{1-v}\sqrt{p+(t+R)}\sqrt{t+u} \sqrt{2+t-u}\sqrt{2-t-u} }\mid_{t=u+v(h-u)}\mathrm{d}v,
\end{split}
\end{equation}
after substituting $t=u+v(h-u)$ in the last step. We have
\begin{equation}
\label{K_diag}
\begin{split}
K_n(\eta ; p,R,h)&=\frac{\sqrt{ph(1-h)}}{2}\int_{0}^{1}\frac{1}{\sqrt{v}\sqrt{1-v}}\mathrm{d}v\\
&=\frac{\pi\sqrt{ph(1-h)}}{2},
\end{split}
\end{equation}
which is non-zero for $h \in [\epsilon, 1-\kappa]$, where $\kappa \in (0,1-\epsilon)$. 

Now that we have a general expression for the Fourier decomposition of $\mathcal{R}f$, we prove injectivity of $\mathcal{R}_L$. Let us define
$$p = p(h)=\frac{\alpha^2+h^2}{2h},\ \ \ \ \ \ R = R(h)=p(h)-h,\ \ \ \text{and}, \ \ \  t = t(h,u) = u+v(h-u).$$
Then,
\begin{equation}
\begin{split}
\widehat{\mathcal{R}_Lf}_n(h,\eta) &= 4\int_{\epsilon}^{h}K_n(\eta ; p(h),R(h),u)\tilde{\hat{f}}_n\paren{u,\eta}\mathrm{d}u, \\
&= \int_{\epsilon}^{h}\widetilde{K_n}(\eta ; h,u)\tilde{\hat{f}}_n\paren{u,\eta}\mathrm{d}u.
\end{split}
\end{equation}
We now aim to prove that $\widetilde{K_n}(\eta,\cdot,\cdot)$, and its first order derivative with respect to $h$, is bounded on $T = \{ (h,u) : \epsilon \leq h \leq 1-\kappa, \epsilon \leq u \leq h\}$, for any fixed $n$ and $\eta$. To do this, we show that all the terms dependent on $h$ under the integral sign on the third line of \eqref{B.9} are bounded and have bounded first order derivative with respect to $h$ within the limits of integration. First, for $h \in [\epsilon, 1-\kappa]$, we have the inequalities
\begin{equation}
\label{bound_1}
%\begin{split}
|p|\leq\frac{1+\alpha^2}{2\epsilon}, \ \ \ \ \ \left|\frac{\mathrm{d}p}{\mathrm{d}h}\right|=\left|\frac{1}{2}-\frac{\alpha^2}{2h^2}\right|\leq \frac{1}{2}+\frac{\alpha^2}{2\epsilon^2},
\end{equation}
and
\begin{equation}
\label{bound_2}
|R|\leq |p| \leq \frac{1+\alpha^2}{2\epsilon}, \ \ \ \ \ \left|\frac{\mathrm{d}R}{\mathrm{d}h}\right|=\left|\frac{1}{2}-\frac{\alpha^2}{2h^2}-1\right|\leq \frac{1}{2}+\frac{\alpha^2}{2\epsilon^2}.
\end{equation}
Further,
\begin{equation}
\label{bound_3}
|t|\leq 1, \ \ \ \text{and} \ \ \left|\frac{\mathrm{d}t}{\mathrm{d}h}\right| =\left|v\right|\leq 1.
%\end{split}
\end{equation}
We have, $\left|\cos\paren{\eta\sqrt{p^2-(t+R)^2}}\right|\leq 1$, and, letting $g_1 = g_1(h,u) = t(h,u)+R(h)$,
\begin{equation}
\label{equ_111}
\begin{split}
&\frac{\mathrm{d}}{\mathrm{d}h}\cos\paren{\eta\sqrt{p^2- g_1^2}}=\\
& -\frac{\eta\left[pp' - g_1 g'_1 \right]} {\sqrt{p^2-g_1^2}} \sin\paren{\eta\sqrt{p^2-g_1^2}}\\
&=-{\eta}^2 \left[ pp'- g_1 g'_1 \right] \sinc\paren{\eta\sqrt{p^2-g_1^2}},
\end{split}
\end{equation}
which is bounded by \eqref{bound_1}-\eqref{bound_3}. In \eqref{equ_111}, $p'$ and $g'_1$ are the derivatives of $p$ and $g_1$, respectively, with respect to $h$.

We have $\left|T_{|n|}\paren{\frac{(1-u)^2+1-t^2}{2(1-u)}}\right|\leq 1$, and, for $n\neq 0$,
\begin{equation}
\begin{split}
\frac{\mathrm{d}}{\mathrm{d}h}T_{|n|}\paren{\frac{(1-u)^2+1-t^2}{2(1-u)}}&=-\frac{v \cdot t}{1-u}T'_{|n|}\paren{\frac{(1-u)^2+1-t^2}{2(1-u)}}\\
&=-\frac{|n|v \cdot t}{1-u}U_{|n|-1}\paren{\frac{(1-u)^2+1-t^2}{2(1-u)}},
\end{split}
\end{equation}
where $U_n$ is a Chebyshev polynomial of the second kind. It follows that
\begin{equation}
\label{ineq}
\left|\frac{\mathrm{d}}{\mathrm{d}h}T_{|n|}\paren{\frac{(1-u)^2+1-t^2}{2(1-u)}}\right|\leq \frac{|n|^2}{\kappa},
\end{equation}
noting $|U_{|n|}|\leq |n|+1$, and $\frac{1}{1-u}\leq \frac{1}{\kappa}$. The $n=0$ case is trivial since $T_0=1$. 

Now, letting $g_2 = g_2(h,u) = p(h)+t(h,u)+R(h)$, we have
$$\frac{1}{\sqrt{p+(t+R)} \sqrt{t+u}} = \frac{1}{\sqrt{g_2}\sqrt{t+u}}\leq\frac{1}{2\epsilon},$$
and
\begin{equation}
\label{equ_123}
\begin{split}
\left| \frac{\mathrm{d}}{\mathrm{d}h}\paren{\frac{1}{\sqrt{g_2}\sqrt{t+u}}} \right|
&= \left| \frac{ (t+u)g_2' + v \cdot g_2}{2 \paren{g_2}^{\frac{3}{2}}\paren{t+u}^{\frac{3}{2}}} \right|\\
&\leq \frac{2|g_2'| + |g_2|}{2(2\epsilon)^3},
\end{split}
\end{equation}
which is bounded, by \eqref{bound_1}-\eqref{bound_3}. In \eqref{equ_123}, $g'_2$ denotes the derivative of $g_2$ with respect to $h$

Finally, we have
$$\frac{1}{\sqrt{2+t-u}\sqrt{2-t-u}} \leq \frac{1}{\sqrt{2\kappa}},$$
and
\begin{equation}
\begin{split}
\left| \frac{\mathrm{d}}{\mathrm{d}h}\frac{1}{\sqrt{2+t-u}\sqrt{2-t-u}}\right| &= \left| \frac{v\cdot t}{ (2+t-u)^{\frac{3}{2}} (2-t-u)^{\frac{3}{2}} } \right| \\
& \leq \paren{\frac{1}{2\kappa}}^{\frac{3}{2}}.
\end{split}
\end{equation}
It follows that the kernel, $\widetilde{K_n}(\eta ; h,u)$, is bounded and has bounded first order derivative with respect to $h$ on $T$. Further, $K_n(\eta ; h,h)\neq 0$, for $h\in[\epsilon, 1-\kappa]$, by \eqref{K_diag}. Thus, using classical theory on Volterra equations \cite{tricomi1985integral}, we can recover $\tilde{\hat{f}}_n\paren{u,\eta}$ uniquely for all $n\in\mathbb{Z}$, $\eta\in\mathbb{R}$, and $u \in [\epsilon, 1]$, after letting $\kappa \to 0$. This proves injectivity of $\Rc_L$ on domain $L^2_c(C_{\epsilon})$. Injectivity of $\Rc$ follows, as $\Rc f = 0 \implies \Rc_L f = 0$. This finishes the proof.
\end{proof}
\end{theorem}

\section{Microlocal analysis}
\label{microlocal_section}

In this section, we analyze the stability of $\Rc$ and $\Rc_L$ from a microlocal perspective. Specifically, we show that $\Rc$ and $\Rc_L$ are elliptic FIO which satisfy the Bolker condition. We also investigate the edge detection capabilities of $\Rc$ and $\Rc_L$. We first analyze the stability of $\Rc_L$ in the following subsection, and then $\Rc$ later in subsection \ref{gen_mic}.

\subsection{Analysis of $\Rc_L$; the $\sqrt{p^2-R^2}=\alpha$ case}
\label{micro_lim}
In this case, the lemon surfaces have the defining function
$$\Psi(\vx; s,\theta_0,z_0)=\frac{|\vx_T|^2-\alpha^2}{|\vx_T'|} - s, \ \ \ s \leq 0$$
where $s=-2R = \paren{h^2-\alpha^2}/{h}$, and $\vx_T=(x-\cos\theta_0,y-\sin\theta_0,z-z_0)^T=(\vx_T',z-z_0)^T$, $\vx'=(x,y)^T$, $\vx=(x,y,z)^T$. $\Psi$ is a defining function in the sense that $\mathcal{L} = \{\Psi = 0\}$ is a lemon surface, when $s\leq 0$. When $s>0$, $\{\Psi=0\}$ is an apple surface (see \cite[figure 4]{rigaud20183d}), and when $s=0$, $\{\Psi=0\}$ is a sphere, radius $\alpha$ with center $(\cos\theta_0,\sin\theta_0,z_0)$. 

For $f\in L_c^2(C_{\epsilon})$, where $\epsilon \in (0,1)$, we have the alternate expression for $\Rc_L$
\begin{equation}
\begin{split}
\mathcal{R}_Lf(s,\theta_0,z_0) &= \int_{C_{\epsilon}}\left|\nabla_{\vx}\Psi\right|\delta\paren{\Psi(\vx; s,\theta_0,z_0)}f(\vx)\mathrm{d}\vx \\
&= \int_{-\infty}^{\infty}\int_{C_{\epsilon}}\left|\nabla_{\vx}\Psi\right|e^{i\sigma\Psi(\vx; s,\theta_0,z_0)}f(\vx)\mathrm{d}\vx\mathrm{d}\sigma,
\end{split}
\end{equation}
where $\Phi=\sigma\Psi$ is the phase function, and $a=\left|\nabla_{\vx}\Psi\right|$ is the amplitude. The weight $\left|\nabla_{\vx}\Psi\right|$ is included so that $\Rc_L f$ defines the integrals of $f$ over lemons with respect to the surface measure on the lemon surface, in line with the theory of \cite{palamodov2012uniform}. Note, we now vary $s$ instead of $h$ as before in the injectivity proofs. We do this to make the calculations in this section easier. We now have our second main theorem.

\begin{theorem}
\label{micro_R_L}
$\Rc_L$ is an elliptic FIO order $-1$ which satisfies the Bolker condition.
\end{theorem}

\begin{proof}
We first prove $\Rc_L$ is an FIO. Clearly, $\Phi=\sigma\Psi$ is homogeneous order one, since $\Psi$ does not depend on $\sigma$. Also, $\mathrm{d}_s\Phi = -\sigma \neq 0$, hence $\mathrm{d}\Phi, \mathrm{d}_{\vy}\Phi \neq \mathbf{0}$. The lemon surfaces are smooth manifolds away from their points of self-intersection (i.e., $(\cos\theta_0,\sin\theta_0,z_0 \pm\alpha)$), which we do not consider as they are outside the support of $f$. Thus, $\Phi$ is clean.

Let $g=|\vx_T'|$. Then, we have
\begin{equation}
\begin{split}
\nabla_{\vx}\Phi&=\sigma\paren{\frac{2\vx_T}{g}-\paren{\frac{|\vx_T|^2-\alpha^2}{g^3}}\vv}\\
&=\frac{\sigma}{g}\paren{\paren{2-\frac{s}{g}}(x-\cos\theta_0),\paren{2-\frac{s}{g}}(y-\sin\theta_0),2(z-z_0)},
\end{split}
\end{equation}
where where $\vv=(x-\cos\theta_0,y-\sin\theta_0,0)$, and $s = \frac{|\vx_T|^2-\alpha^2}{|\vx_T'|}$. Thus,
\begin{equation}
\left|\nabla_{\vx}\Phi\right| = \sigma\sqrt{ \paren{2-\frac{s}{g}}^2 + \frac{4(z-z_0)^2}{g^2} }.
\end{equation}
Now,
$$2-\frac{s}{g}=\frac{(x-\cos\theta_0)^2+(y-\sin\theta_0)^2-(z-z_0)^2+\alpha^2}{g^2},$$
which is zero if and only if $(x-\cos\theta_0)^2+(y-\sin\theta_0)^2-(z-z_0)^2+\alpha^2=0$. The two-sheeted hyperboloid $H=\{(x-\cos\theta_0)^2+(y-\sin\theta_0)^2-(z-z_0)^2+\alpha^2=0\}$ intersects the lemon only at its points of self-intersection (i.e., $(\cos\theta_0,\sin\theta_0,z_0 \pm\alpha)$), which we do not consider as they are outside the support of $f$.  It follows that $\mathrm{d}_{\vx} \Phi \neq \mathbf{0}$, and $\Phi$ is non-degenerate. Further, $a = \left|\nabla_{\vx}\Phi\right| >0$, and hence $a$ is an elliptic symbol. $a$ is order zero since it is smooth and does not depend on $\sigma$. Hence, by Definition \ref{FIOdef}, $\Rc_L$ is an elliptic FIO order $\mathcal{O}(\Rc_L) = 0 +\frac{1}{2} - \frac{3}{2} = -1$. 

Let $Y_1 = \dot{\mathbb{R}}\times [0,2\pi] \times \mathbb{R}$. The left projection $\Pi_L : C_{\epsilon} \times Y_1 \to \Pi_L\paren{ C_{\epsilon} \times Y_1 }$ of $\Rc_L$ is defined
\begin{equation}
\label{pi_L}
\Pi_L(\vx; \sigma,\theta_0,z_0)=\paren{\overbrace{-\sigma}^{\mathrm{d}_s\Phi},\theta_0,z_0,\overbrace{\frac{|\vx_T|^2-\alpha^2}{g}}^{s},\overbrace{\frac{\sigma\paren{\Theta_0^{\perp}\cdot\vx_T'}}{g}\left[2-\frac{s}{g}\right]}^{\mathrm{d}_{\theta_0}\Phi}, \overbrace{\frac{2\sigma(z-z_0)}{g}}^{\mathrm{d}_{z_0}\Phi}},
\end{equation}
where $\Theta_0=(\cos\theta_0,\sin\theta_0)^T$ and $\Theta_0^{\perp}=(\sin\theta_0,-\cos\theta_0)^T$. Note, $g>\epsilon$ on the support of $f$, so we are never dividing by zero in \eqref{pi_L}. By Definition \ref{def:bolker}, $\Rc_L$ satisfies the Bolker condition if $\Pi_L$ is an injective immersion. We now break the remainder of the proof into two parts. First we prove injectivity of $\Pi_L$, and then we prove that $\Pi_L$ is an immersion.\\
\\
\textbf{Injectivity.}
Let $\vx_i = (x_i,y_i,z_i)$, for $i=1,2$, and let $\Pi_L(\vx_1; \sigma,\theta_0,z_0)=\Pi_L(\vx_2; \sigma,\theta_0,z_0)$. Then $\vx_1$ and $\vx_2$ lie on the same lemon, using the $s$ term of \eqref{pi_L}. We also have $(z_1-z_0)/g_1=(z_2-z_0)/g_2$, where $g_i = \sqrt{(x_i-\cos\theta_0)^2+(y_i-\sin\theta_0)^2}$.
Let $z'=z-z_0$. For $\vx$ on a lemon, center $(\cos \theta_0,\sin\theta_0,z_0)$, $h(z')=z'/g(z')$ is a strictly monotone increasing function of $z'$, for $z'\in (-\alpha,\alpha)$. Indeed, $g(z')=\sqrt{R^2+\alpha^2-z'^2}-R$, and
\begin{equation}
\begin{split}
\frac{\mathrm{d}}{\mathrm{d}z'}h(z')&=\frac{R^2+\alpha^2-R\sqrt{R^2-z'^2+\alpha^2}}{\sqrt{R^2-z'^2+\alpha^2}\paren{\sqrt{R^2-z'^2+\alpha^2}-R}^2}\\
&>\frac{p^2-Rp}{p\paren{p-R}^2}\\
&=\frac{1}{p-R}>0, \ \text{for}\ z'\in (-\alpha,\alpha),
\end{split}
\end{equation}
noting $p^2=R^2+\alpha^2$. See figure \ref{F3}, for an example $h(z')$ plot.
\begin{figure}[!h]
\centering
\includegraphics[width=0.9\linewidth, height=5cm, keepaspectratio]{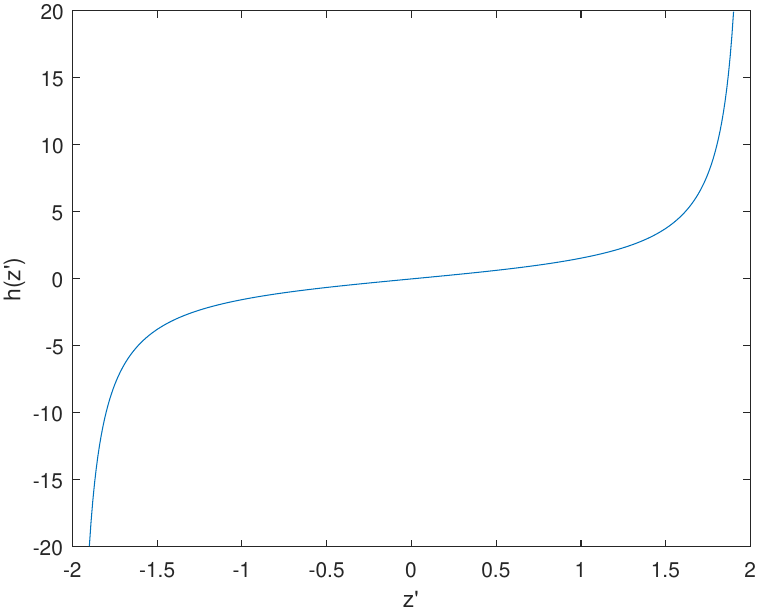}
\caption{Plot of $h(z')$, on $(-\alpha,\alpha)$, when $R, \alpha = 2$. }
\label{F3}
\end{figure}
Thus, $h(z'_1)=h(z'_2) \implies z_1=z_2$, which further implies $g_1=g_2$. By the same arguments as above, towards the start of the proof, $2-\frac{s}{g}\neq0$, and it follows that 
$$\Theta_0^{\perp}\cdot(x_1-\cos\theta_0, y_1-\sin\theta_0)^T=\Theta_0^{\perp}\cdot(x_2-\cos\theta_0, y_2-\sin\theta_0)^T.$$
Thus, either $\vx_1$ and $\vx_2$ are equal, or they are the reflections of one another in the plane $(\Theta_0^T,0)\cdot\vx_T=0$, i.e., the plane tangent to the boundary of $C_0$ at $(\cos\theta_0,\sin\theta_0,0)$. In the latter case, one of $\vx_1$ or $\vx_2$ must lie outside $C_0$, given the convexity of $C_0$ boundary. Therefore, we conclude that $\Pi_L$ is injective, since $f$ is supported on $C_{\epsilon}\subset C_0$. \\
\\
\textbf{Immersion.}
Let $I'_{3 \times 3} = \text{diag}(-1,1,1)$. Then, the derivative of $\Pi_L$ is
\begin{equation}\label{DPiA}
D\Pi_L=\kbordermatrix {&\mathrm{d}\sigma,\mathrm{d}\theta_0,\mathrm{d}z_0 & \nabla_\vx \\
-\sigma,\theta_0,z_0 & I'_{3\times 3} & \textbf{0}_{3 \times 3} \\
s& \cdot & \vr_1\\
\mathrm{d}_{z_0} \Phi & \cdot & \vr_2\\ 
{\mathrm{d}_{\theta_0} \Phi} & \cdot & \vr_3}.
\end{equation}
The $\vr_i$, for $1\leq i \leq 3$, are defined as
\begin{equation}
\begin{split}
\vr_1=\nabla_{\vx}(s)&=\frac{2\vx_T}{g}-\paren{\frac{|\vx_T|^2-\alpha^2}{g^3}}\vv\\
&=\frac{1}{g}\paren{\paren{2-\frac{s}{g}}(x-\cos\theta_0),\paren{2-\frac{s}{g}}(y-\sin\theta_0),2(z-z_0)},
\end{split}
\end{equation}
and
\begin{equation}
\begin{split}
\vr_2=\nabla_{\vx}\paren{\mathrm{d}_{z_0}\Phi}&=\frac{2\sigma}{g}\paren{-\frac{(z-z_0)(x-\cos\theta_0)}{g^2},-\frac{(z-z_0)(y-\sin\theta_0)}{g^2},1},
\end{split}
\end{equation}
and
\begin{equation}
\begin{split}
\vr_3=\nabla_{\vx}\paren{\mathrm{d}_{\theta_0}\Phi}&=\frac{\sigma}{g}\paren{2-\frac{s}{g}}(\sin\theta_0,-\cos\theta_0,0)+\frac{\sigma\paren{\Theta_0^{\perp}\cdot \vx_T'}}{g^3}\left[- \vv + \frac{\vu}{g^2}\right],
\end{split}
\end{equation}
where $\vv=(x-\cos\theta_0,y-\sin\theta_0,0)$, and
$$\vu=\paren{3\left[(z-z_0)^2-\alpha^2\right](x-\cos\theta_0),3\left[(z-z_0)^2-\alpha^2\right](y-\sin\theta_0),-2g^2(z-z_0)},$$
noting
$$\frac{1}{g}\paren{2-\frac{s}{g}} = \frac{1}{g} - \frac{(z-z_0)^2-\alpha^2}{g^3}$$
in the calculation of $\vr_3$.
$D\Pi_L$ is invertible if and only if the matrix $M=\paren{\vr_1^T,\vr_2^T,\vr_3^T}$ is invertible. We have
$$\vr_1 \times \vr_2 = \frac{2\sigma\paren{|\vx_T|^2+\alpha^2}}{g^4}\paren{-(y-\sin\theta_0),(x-\cos\theta_0),0}.$$
Thus
$$\text{det}(M)=\vr_3\cdot (\vr_1 \times \vr_2 )=-\paren{2-\frac{s}{g}} \cdot \frac{2\sigma^2\paren{|\vx_T|^2+\alpha^2}}{g^5}\paren{(\Theta_0^T,0)\cdot \vx_T},$$
which is zero if and only if $2-\frac{s}{g}=0$ or $(\Theta_0^T,0)\cdot \vx_T=0$. $2-\frac{s}{g}$ was shown earlier to be never zero for points on the lemon, except the singular points, which we do not consider. The plane $\{(\Theta_0^T,0)\cdot \vx_T=0\}$ does not intersect $C_{\epsilon}$,  by convexity of the $C_0$ boundary. Hence $(\Theta_0^T,0)\cdot \vx_T\neq 0$ for $\vx \in C_{\epsilon}$. Thus, $\Pi_L$ is an immersion.
\end{proof}

\begin{remark}
The above theorem shows that any reconstruction artifacts due to the Bolker condition are reflections of the true singularities in planes tangent to $C_0$. Since $f$ is supported on $C_0$, this means, due to the convexity of the $C_0$ boundary, any artifacts due to Bolker must lie in $\mathbb{R}^3\backslash C_0$, and thus do not interfere with the region of interest. If $f$ were supported outside of $C_0$, e.g., if $\text{supp}(f)\subset \Omega \subset \mathbb{R}^3\backslash C_0$, where $\Omega$ is an open subset, then the Bolker condition would fail. This is true since, for any point $\vx\in \mathbb{R}^3\backslash C_0$, we can find a plane tangent to $C_0$ which intersects $\vx$, and thus the immersion part of Bolker would fail for $f$ supported outside of $C_0$.
\end{remark}

\subsection{Analysis of $\Rc$; the general case}
\label{gen_mic}
In this section, we use the following defining function for the lemon surfaces
$$\Psi(\vx; t,R,\theta_0,z_0)=\paren{\sqrt{(x-\cos\theta_0)^2+(y-\sin\theta_0)^2}+R}^2+(z-z_0)^2-t,$$
where $t=p^2$. We use a different definition for $\Psi$ than in section \ref{micro_lim} here for ease of calculation.

Similarly to the last section, for $f\in L_c^2(C_{\epsilon})$, we have the alternate expression for $\Rc$
\begin{equation}
\begin{split}
\mathcal{R}f(t,R,\theta_0,z_0) &= \int_{C_{\epsilon}}\left|\nabla_{\vx}\Psi\right|\delta\paren{\Psi(\vx; t,R,\theta_0,z_0)}f(\vx)\mathrm{d}\vx \\
&=  \int_{-\infty}^{\infty}\int_{C_{\epsilon}}\left|\nabla_{\vx}\Psi\right|e^{i\sigma\Psi(\vx; t,R,\theta_0,z_0)}f(\vx)\mathrm{d}\vx\mathrm{d}\sigma,
\end{split}
\end{equation}
where $\Phi=\sigma\Psi$ is the phase and $a=\left|\nabla_{\vx}\Psi\right|$ is the amplitude. We now have our third main theorem. The proof uses similar ideas to Theorem \ref{micro_R_L}.

\begin{theorem}
\label{micro_R}
$\Rc$ is an elliptic FIO order $-3/2$ which satisfies the Bolker condition. 
\end{theorem}

\begin{proof}
We first prove $\Rc$ is an FIO. $\Phi=\sigma\Psi$ is homogeneous order one, since $\Psi$ does not depend on $\sigma$. $\mathrm{d}_t\Phi = -\sigma \neq 0$, hence $\mathrm{d}\Phi, \mathrm{d}_{\vy}\Phi \neq \mathbf{0}$.  $\Phi$ is clean for the same reasons as in the proof of theorem \ref{micro_R_L}.

Let $\ve_3 = (0,0,1)$, and let $\vv=(x-\cos\theta_0,y-\sin\theta_0,0)$. Then, we have
\begin{equation}
\begin{split}
\nabla_{\vx}\Phi&=2\sigma \paren{-\frac{(g+R)}{g^3} \cdot \vv+(z-z_0)\ve_3},
\end{split}
\end{equation}
where $g$ is defined as in the proof of Theorem \ref{micro_R_L}, and,
\begin{equation}
\left|\nabla_{\vx}\Phi\right| = 2\sigma\sqrt{ \frac{(g+R)^2}{g^4} + (z-z_0)^2 } > 0.
\end{equation}
Thus, $\mathrm{d}_{\vx} \Phi \neq \mathbf{0}$, and $\Phi$ is non-degenerate. Further, $a = \left|\nabla_{\vx}\Phi\right| > 0$, and hence $a$ is an elliptic symbol. $a$ is order zero since it is smooth and does not depend on $\sigma$. Hence, by Definition \ref{FIOdef}, $\Rc$ is an elliptic FIO order $\mathcal{O}(\Rc) = 0 +\frac{1}{2} - \frac{4}{2} = -\frac{3}{2}$. 

Let $Y_2 = \dot{\mathbb{R}} \times (0,\infty) \times [0,2\pi] \times \mathbb{R}$. The left projection $\Pi_L : C_{\epsilon} \times Y_2 \to \Pi_L\paren{C_{\epsilon} \times Y_2}$ of $\Rc$ is defined
\begin{equation}
\begin{split}
\label{pi_L_1}
\Pi_L(\vx; \sigma,R,\theta_0,z_0)&=\Bigg(\overbrace{-\sigma}^{\mathrm{d}_t\Phi},R,\theta_0,z_0,\overbrace{\paren{g+R}^2+(z-z_0)^2}^{t},\\
&\ \ \ \  \ \ \ \ \  \ \ \ \ \ \  \ \ \ \ \underbrace{2\sigma\paren{g+R}}_{\mathrm{d}_{R}\Phi},\underbrace{2\sigma\paren{\Theta_0^{\perp}\cdot\vx'_T}\paren{1+\frac{R}{g}}}_{\mathrm{d}_{\theta_0}\Phi}, \underbrace{2\sigma(z-z_0)}_{\mathrm{d}_{z_0}\Phi}\Bigg).
\end{split}
\end{equation}
As in the proof of Theorem \ref{micro_R_L}, we split the remainder of the proof into injectivity and immersion parts.\\
\\
\textbf{Injectivity.}
Let $\Pi_L(\vx_1; \sigma,R,\theta_0,z_0)=\Pi_L(\vx_2; \sigma,R,\theta_0,z_0)$. Then $z_1=z_2$, and $g_1=g_2$. Since $1+\frac{R}{g}>0$, $\Theta_0^{\perp}\cdot(x_1-\cos\theta_0, y_1-\sin\theta_0)^T=\Theta_0^{\perp}\cdot(x_2-\cos\theta_0, y_2-\sin\theta_0)^T$. Thus, as in the proof of Theorem \ref{micro_R_L}, $\vx_1$ and $\vx_2$ are the reflections of one another in the plane $\{(\Theta_0^T,0)\cdot\vx_T=0\}$, or they are equal. We conclude that $\Pi_L$ is injective, given the convexity of the boundary of $C_0$. \\
\\
\textbf{Immersion.}
Let $I'_{4\times 4} = \text{diag}\paren{-1,1,1,1}$. Then, the differential of $\Pi_L$ is
\begin{equation}\label{DPiA_1}
D\Pi_L=\kbordermatrix {&\mathrm{d}\sigma, \mathrm{d}R,\mathrm{d}\theta_0,\mathrm{d}z_0 & \nabla_\vx \\
-\sigma,R,\theta_0,z_0 & I'_{4\times 4} & \textbf{0}_{4 \times 3} \\
\mathrm{d}_{\theta_0} \Phi & \cdot & \vr_1 \\
t & \cdot & \vr_2 \\
\mathrm{d}_{z_0} \Phi & \cdot & \vr_3 \\ 
{\mathrm{d}_{R} \Phi} & \cdot & \vr_4 }.
\end{equation}
The $\vr_i$, for $1\leq i \leq 4$, are defined
\begin{equation}
\vr_1 = \nabla_{\vx}\paren{ \mathrm{d}_{\theta_0} \Phi } = 2\sigma\paren{-\frac{R}{g^3}\paren{\Theta_0^{\perp}\cdot\vx'_T} \vv + \paren{1+\frac{R}{g}}\paren{\Theta_0^{\perp},0} },
\end{equation}
and
\begin{equation}
\vr_2 = \nabla_{\vx}\paren{ t } = 2\paren{-\frac{(g+R)}{g^3} \cdot \vv+(z-z_0)\ve_3},
\end{equation}
$\vr_3 = \nabla_{\vx}\paren{ \mathrm{d}_{z_0} \Phi } = 2\sigma \ve_3$, and $\vr_4 = \nabla_{\vx}\paren{ \mathrm{d}_{R} \Phi } = -2\sigma \frac{\vv}{g^3}$
Let $M = \left[\vr_1^T,\vr_2^T,\vr_3^T\right]$.
%\begin{equation}
%M = \begin{pmatrix}
%-\frac{t}{g^3}\paren{\Theta_0^{\perp}\cdot\vx'_T} \vv + \paren{1+\frac{t}{g}}\paren{\Theta_0^{\perp},0} \\
%-\frac{\vv}{g^3}(g+t) +(z-z_0)\ve_3\\
%\ve_3\\
%\end{pmatrix}.
%\end{equation}
Then $\text{det} M \neq 0 \implies \text{det} D\Pi_L \neq 0$. Let $a = \frac{R}{g^3}\paren{\Theta_0^{\perp}\cdot\vx'_T}$, $b=1+\frac{R}{g}$, and $c = -\frac{g+R}{g^3}$. Then,
\begin{equation}
\begin{split}
\text{det} M &= \text{det} \begin{pmatrix}
2\sigma \big( a(x-\cos\theta_0) + b\sin\theta_0 \big) &  2\sigma \big( a(y-\sin\theta_0) - b\cos\theta_0 \big) & 0  \\
2 c(x-\cos\theta_0) & 2 c(y-\sin\theta_0) & 2 (z-z_0)\\
0 & 0 & 2\sigma
\end{pmatrix}\\
&= 8\sigma^2 \cdot bc \left[ \cos\theta_0(x-\cos\theta_0)+\sin\theta_0(y-\sin\theta_0) \right]\\
&= -8\sigma^2 \cdot \frac{\paren{g+R}^2}{g^4}\left[ \paren{\Theta_0^T,0} \cdot \vx - 1 \right],
\end{split}
\end{equation}
which is non-zero unless $\paren{\Theta_0^T,0} \cdot \vx = 1$. This happens on planes tangent to the boundary of $C_0$, which do not intersect $C_0$. Thus, $\Pi_L$ is an immersion.
\end{proof}

\subsection{Visible singularities and edge detection}
\label{visible}

In this section, we investigate the edge detection capabilities of $\mathcal{R}$ and $\Rc_L$ using microlocal analysis. See \cite{krishnan2015microlocal} for similar work, where the authors present a microlocal analysis of the classical hyperplane Radon transform, commonly applied in X-ray CT. 
%$\mathcal{R}f$ defines the integrals of $f$ over lemon surfaces. 

Let $(\vx,\vxi) \in \text{WF}(f)$ be a point in the wavefront set (edge map) of $f$, where $\vx \in C_0$ is the location of the edge, and $\vxi \in S^2$ is the direction, where $S^2$ is the unit sphere in $\mathbb{R}^3$. Edges which are intersected by, and in directions normal to a lemon in our data set are detectable, or visible. Otherwise, we say the edge is undetectable, or invisible.
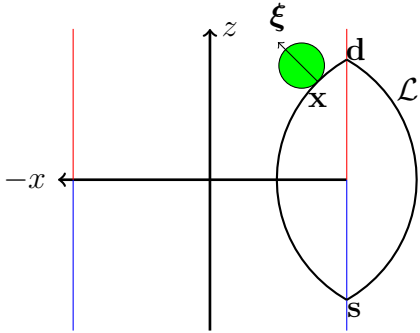
\begin{figure}[!h]
\centering
\begin{tikzpicture}[scale=0.4]
%\draw [fill=green] (1.5113,0.9299) circle (0.75);
%\node at (2.6,0.8) {$\vx$};
%\draw [->] (2.25,1-0.2)--(0.5264,1.1030)node[above]{$\xi$};
\draw [fill=green] (3.0170,3.7830) circle (0.75);
\coordinate (X) at (3.5473,3.2527);
\coordinate (xi) at (2.3099,4.4901);
\draw [->] (3.5473,3.2527)node[below] {$\vx$}--(2.2392,4.5608) node[above]{$\vxi$};
%\node at (0.5,2) {$f$};
\draw [<-,line width=1pt] (-5,0)node[left] {$-x$}--(4.5,0);
\draw [->,line width=1pt] (0,-5)--(0,5)node[right] {$z$};
%\draw (4.5,0)--(6.8,0);
%\node at (4.5+1.15,-0.3) {$R$};
%\draw (4.5,3.9837)--(6.8,0);
%\node at (4.5+1.5,2.1) {$p$};
\coordinate (x) at (3,2.6458);
\coordinate (O) at (0,0);
\coordinate (c) at (4.5,0);
%\draw (O)--(x);
%\draw (c)--(x);
%\draw (x)--(4.5,2.6458);
%\node at (3.9,3) {$t$};
%\node at (4.8,1.25) {$z$};
\draw [blue] (4.5,-5)--(4.5,0);
\draw [red] (4.5,0)--(4.5,5);
\draw [blue] (-4.5,-5)--(-4.5,0);
\draw [red] (-4.5,0)--(-4.5,5);
\draw [thick,domain=120:240] plot ({6.8+4.6*cos(\x)}, {4.6*sin(\x)});
\draw [thick,domain=-60:60] plot ({-6.8+2*4.5+4.6*cos(\x)}, {4.6*sin(\x)});
\node at (4.75,-4.3) {$\vs$};
\node at (4.8,4.35) {$\vd$};
\node at (6.5,3) {$\mathcal{L}$};
\coordinate (D) at (4.5,4);
\coordinate (w) at (2.3,1);
\coordinate (a) at (1.8,2);
\end{tikzpicture}
\caption{An edge, at position $\vx$ in direction $\vxi$, on the boundary of a green disc, normal to a lemon, $\mathcal{L}$.}
\label{fig3}
\end{figure}
See figure \ref{fig3}. The highlighted edge at $\vx$ in direction $\vxi$ is normal to $\mathcal{L}$ and is thus detectable. If an edge is detectable, then it can be stably reconstructed. Edges which are not detectable are invisible to the data and cannot be recovered stably \cite{krishnan2015microlocal}, without sufficient a-priori information regarding the edge map of $f$. Thus, the edge detection capabilities of $\mathcal{R}$ and $\Rc_L$ give insight into the inversion stability.
\begin{figure}[!h]
\centering
\begin{subfigure}{0.32\textwidth}
\includegraphics[width=0.9\linewidth, height=4cm, keepaspectratio]{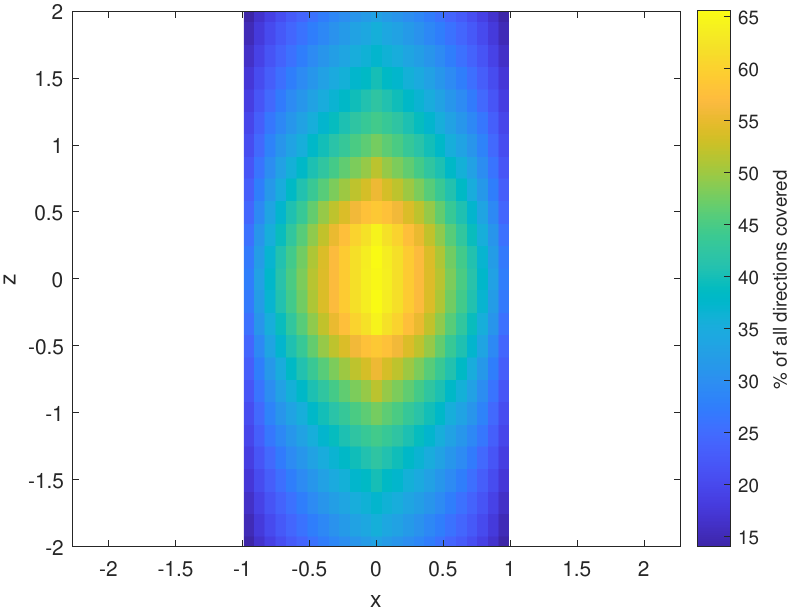}
\subcaption{$\Rc_L f$ edge detection, with $\alpha = 2$} \label{Fvis_2a}
\end{subfigure}
\begin{subfigure}{0.32\textwidth}
\includegraphics[width=0.9\linewidth, height=4cm, keepaspectratio]{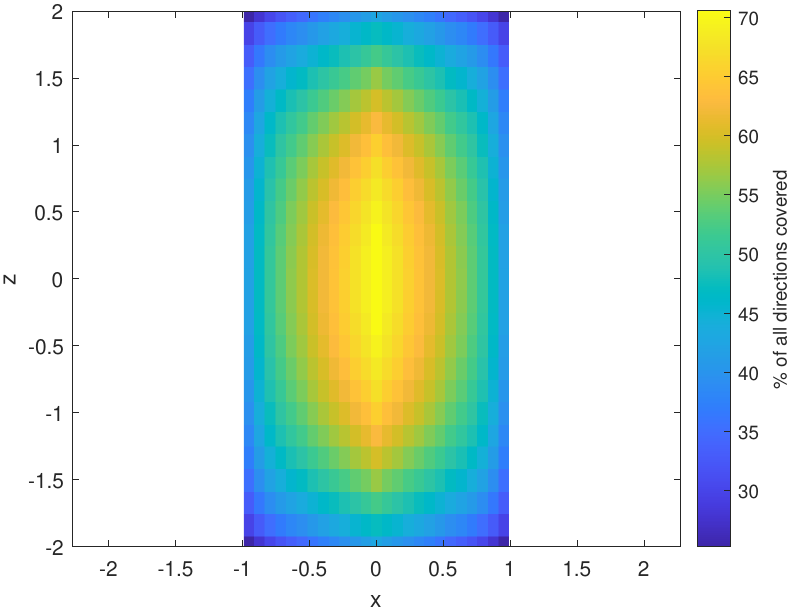} 
\subcaption{$\Rc f$ edge detection} \label{Fvis_2b}
\end{subfigure}
%\begin{subfigure}{0.32\textwidth}
%\includegraphics[width=0.9\linewidth, height=4cm, keepaspectratio]{cart_3D_1}
%\subcaption{spherical - cropped} \label{Fvis_2c}
%\end{subfigure}
\caption{Visible singularities within $C_0$, cropped to height 4. 
%%The lemon surfaces considered to form the edge detection plot are constrained to have centers on the cropped $C_0$, i.e., the $z_0$ coordinate of the lemon centers satisfies $|z_0|\leq 2$.
%The sources and detectors are positioned at $1^{\circ}$ uniform intervals around the cropped $C_0$. The conveyor positions are $z_0 \in\{-2+\frac{4j}{N-1} : 0\leq j\leq N-1\}$, where $N = 100$.
}
\label{Fvis_2}
\end{figure}

We now quantify the edge detection capabilities of $\mathcal{R}$ and $\Rc_L$. See figure \ref{Fvis_2}, where we have plotted the distribution of detectable edges within $C_0$. Similar plots are presented in \cite[section 4.2]{webber2022ellipsoidal} relating to a spheroid Radon transform. For each point, $\vx \in C_0$, we measure the proportion of directions $\vxi$ on the whole sphere surrounding $\vx$ which can be detected by lemon integral data. We compare the edge detection capabilities of $\mathcal{R}$ and $\mathcal{R}_L$. The percentage of detectable edges on the colorbars in figures \ref{Fvis_2a} and \ref{Fvis_2b} ranges from 0 to 100, with larger values indicating greater edge detection, and vice-versa. In figures \ref{Fvis_2a} and \ref{Fvis_2b}, the proportion of detectable edges is greater towards the cylinder center and this reduces significantly near the boundary of $C_0$ and towards the corners of the image.
%By construction, $\Rc$ will detect more edges than $\mathcal{R}_L$ as $\Rc f$ determines $\Rc_L f$. 
We can see that $\Rc$ detects more edges than $\Rc_L$, as the values on the colorbar in figure \ref{Fvis_2b} are closer to $100\%$ when compared to those in figure \ref{Fvis_2a}, and more edges are covered away from the cylinder center, e.g., the yellow region in figure \ref{Fvis_2b} is larger than the yellow region in figure \ref{Fvis_2a}. This makes the inversion of $\Rc$ more stable, when compared to $\Rc_L$. As discussed earlier near the start of section \ref{section_spindle}, to measure $\Rc f$, we require multiple rings of detectors and sources. To measure $\Rc_L f$, we only need one source ring and one detector ring. This makes the measurement of $\Rc_L f$ cheaper, more efficient, and more practical, when compared to $\Rc f$, at the cost of reduced inversion stability and less data.

In figures \ref{Fvis_2a} and \ref{Fvis_2b}, nowhere in $C_0$ are $100\%$ of edges detectable. The distribution of missing edges is also non-uniform, and the missing edges are concentrated near directions parallel to the $z$ axis. This is due to the nature of the cylindrical scanner, and the specific geometry of the lemon surfaces. For example, for any $\vx \in C_0$ and direction $\vxi$, parallel to $(0,0,1)$, there does not exist a lemon in our data set (either $\Rc_L f$ or $\Rc f$) which intersects $\vx$ normal to $\vxi$, and thus the problem is one of limited-angle tomography \cite{krishnan2015microlocal}. If the inversion of $\mathcal{R}$ and $\Rc_L$ is not sufficiently regularized, we will likely see a blurring effect in the reconstruction, particularly near edges with direction parallel to, or a small angle from the $z$ direction. Similar blurring effects are observed, e.g., in conventional limited-angle X-ray CT \cite{krishnan2015microlocal}.

\section{Image reconstructions}
\label{results}
In this section, we present simulated reconstructions from $R_L f$ data when $\alpha = 2$ and $C_0$ is cropped to have height 4 as in section \ref{visible}. We show reconstructions of delta functions, using the Landweber method, to validate our microlocal theory. We then present simulated reconstructions of image phantoms with added Gaussian noise. We first explain our data simulation and the image reconstruction methods used.

\subsection{Reconstruction methods and data simulation}
\label{recon_mthd}
In this subsection, we explain how we simulate data and detail the image reconstruction methods used.
Let $A$ denote the discretized $R_L$ operator, let $\vx$ denote a vectorized image, and let $\vb$ be the measured data. Then, we simulate $\vb$ using
\begin{equation}
\vb = A\vx + \gamma \times \frac{\|A\vx\|_2}{\sqrt{l}} \mathbf{\eta},
\end{equation}
where $l$ is the length of $A\vx$, $\mathbf{\eta} \sim \mathcal{N}(0,1)$ is a vector, size $l\times 1$, of draws from a standard Gaussian, and $\gamma$ is a parameter which controls the noise level. Then, $\|A\vx-\vb\|/\|A\vx\| \approx \gamma$, for large $l$. The image resolution used throughout this section is $41\times 41\times 41$, and we simulate $l = 26691$ lemon integrals to reconstruct $\vx$. Specifically, $\Rc_L f(h,\theta_0,z_0)$ is sampled for $(h,\theta_0,z_0)\in [0,2]\times [0,2\pi]\times [-3,3]$ on a meshgrid with $l$ pixels. $A$ has size $26691 \times 41^3$, and is underdetermined. $A$ is also sparse, and we store $A$ as a sparse matrix.

To recover $\vx$ from $\vb$, we use the Landweber method and a Conjugate Gradient Least Squares (CGLS) and TV hybrid algorithm which are detailed below:
\begin{enumerate}
\item Landweber - to implement the Landweber method \cite{landweber1951iteration}, we apply the code provided in \cite{hansen2018air}, with the default settings.
\item CGLS-TV hybrid - we use a combination of the Non-Negative CGLS (NNCGLS) code of \cite{gazzola2019ir} and the 3-D TV denoising code of \cite{perraudin2014unlocbox}; specifically we use the Matlab function ``prox\_tv3d" provided in \cite{perraudin2014unlocbox} to implement TV smoothing. We run $m_1$ iterations of NNCGLS, followed by $m_2$ iterations of TV smoothing. We then repeat the whole process $m$ times to obtain a reconstruction.
\end{enumerate}
The Landweber method offers modest regularization, and is included to validate our microlocal theory and highlight the artifacts we expect to see if the solution is not sufficiently regularized. The CGLS-TV hybrid method is included to show the performance of more powerful regularizers, such as TV and non-negativity. The assumption of non-negativity is appropriate, as the reconstruction target in CST is an electron density, which is non-negative. To quantify method performance, we use the relative least squares error
\begin{equation}
\epsilon_r = \frac{\|\vx-\vx_{\epsilon}\|_2}{\|\vx\|},
\end{equation}
where $\vx$ is the ground truth, and $\vx_{\epsilon}$ is a reconstruction. In the implementation of CGLS-TV, we choose the TV smoothing parameter so as to minimize $\epsilon_r$. We choose the number of Landweber iterations in the same way.

\begin{figure}[!h]
\centering
\begin{subfigure}{0.4\textwidth}
\includegraphics[width=1.2\linewidth, height=4.5cm, keepaspectratio]{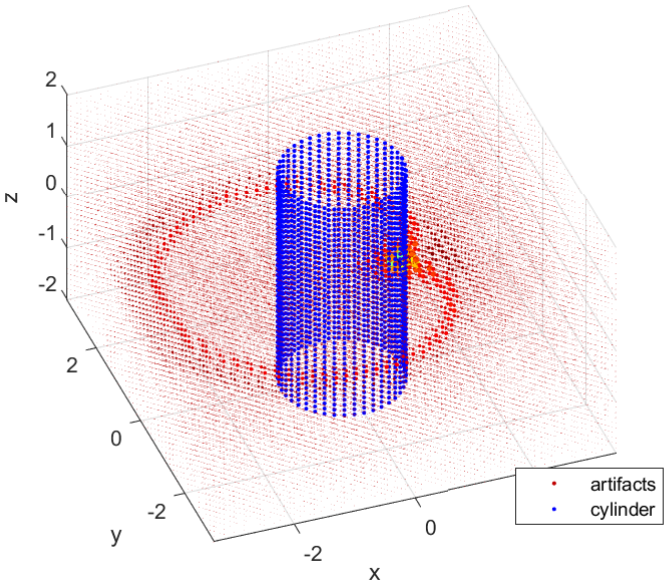} 
\subcaption{observed artifacts, 3-D} \label{bolker_3D_a}
\end{subfigure}
\begin{subfigure}{0.421\textwidth}
\includegraphics[width=1.2\linewidth, height=4.5cm, keepaspectratio]{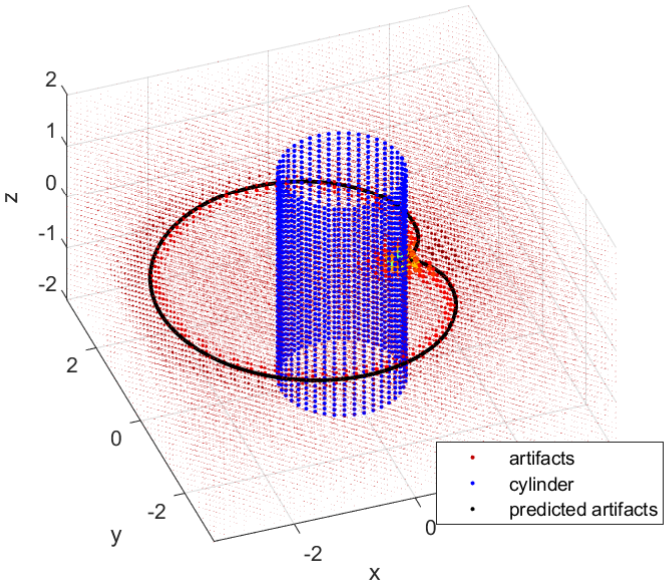}
\subcaption{predicted artifacts, 3-D} \label{bolker_3D_b}
\end{subfigure}
\begin{subfigure}{0.4\textwidth}
\includegraphics[width=1.2\linewidth, height=4.5cm, keepaspectratio]{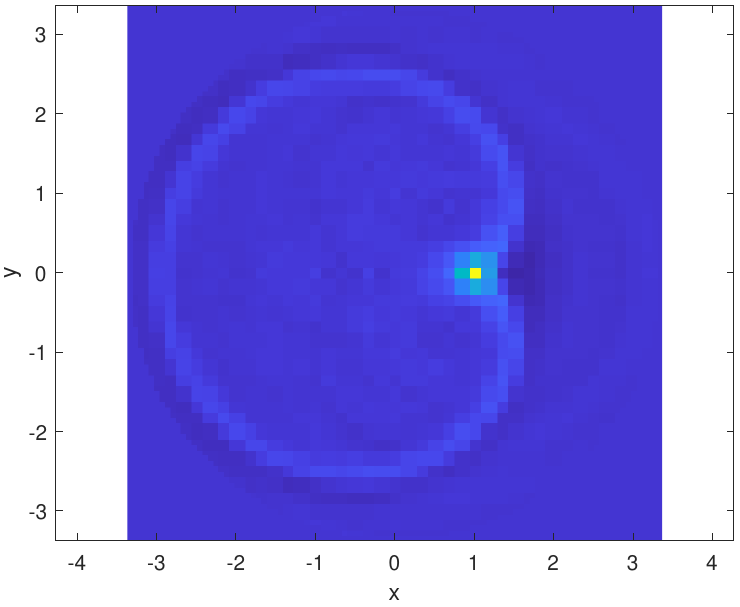} 
\subcaption{observed artifacts, 2-D} \label{bolker_2D_a}
\end{subfigure}
\begin{subfigure}{0.421\textwidth}
\includegraphics[width=1.2\linewidth, height=4.65cm, keepaspectratio]{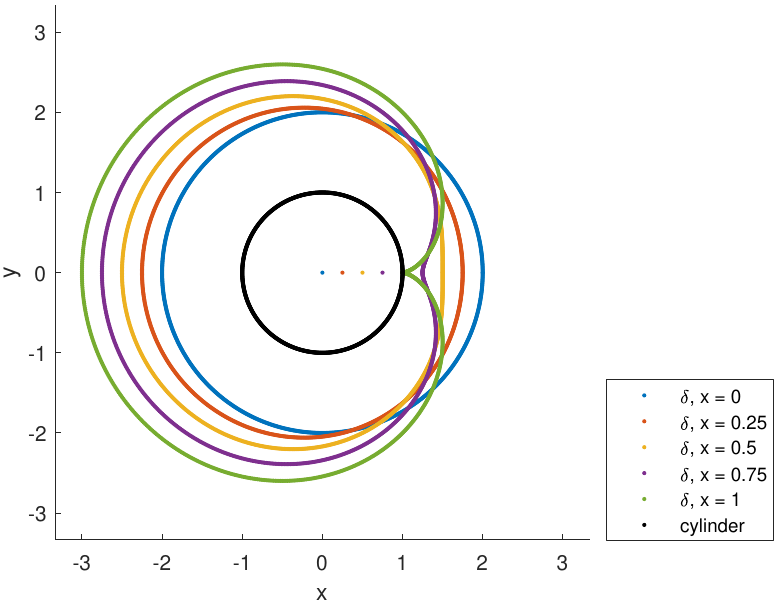}
\subcaption{predicted artifacts, 2-D} \label{bolker_2D_b}
\end{subfigure}
\caption{Comparison of observed and predicted artifacts due to Bolker. (A) - 3-D reconstruction of a delta function at position $\vx = (1,0,0)$ using the Landweber method. (B) - predicted artifacts due to Bolker based on Theorem \ref{micro_R_L}. The black curve represents the artifacts predicted by our theory, which passes directly through the observed red artifact curve in a cardioid shape. (C) - $(x,y)$ plane cross-sectional slice of the image in (A). (D) - predicted artifacts due to Bolker in $(x,y)$ plane. The dots in (D) are delta function positions. We show the predicted artifact curves for varying delta positions. The color of the artifact curve matches the color of the dot where the delta function is located. The cardiod curve in the reconstruction in (C) matches with the green predicted curve in (D), which corresponds to the delta function with position $\vx = (1,0,0)$.}
\label{bolker_3D}
\end{figure}

\subsection{Example artifacts due to Bolker}
In this subsection, we present reconstructions of a delta function to validate our microlocal theory. The noise level is set to $\gamma = 0$, in this subsection. The delta function is positioned at $\vx = (1,0,0)$, on the boundary of $C_0$, and reconstructed using the Landweber method. See figure \ref{bolker_3D}. 
In figure \ref{bolker_3D_a}, we show a three-dimensional image reconstruction of a delta function at position $\vx = (1,0,0)$. The red dots represent the delta function reconstruction, and the blue dots show the boundary of $C_0$. The sizes of the red dots correspond to image density, with larger dots indicating higher density, and vice-versa. We see a high density around $\vx = (1,0,0)$, at the location of the delta function, and also along a cardoid curve which wraps around $C_0$. In figure \ref{bolker_3D_b}, we show the artifacts due to Bolker predicted by our theory. The predicted artifact curve is drawn in black and passes directly through the observed artifact curve in figure \ref{bolker_3D_a}. In figure \ref{bolker_2D_a}, we show the $(x,y)$ plane cross-sectional slice of the image in figure \ref{bolker_3D_a}. In figure \ref{bolker_2D_b}, we show the predicted artifact curves for varying delta function positions in the $(x,y)$ plane. The green curve in figure \ref{bolker_2D_b}, which corresponds to delta position $\vx = (1,0,0)$, matches with the observed artifacts in figure \ref{bolker_2D_a}.

In Theorem \ref{micro_R_L}, we showed that any artifacts due to Bolker are reflections of the true singularities in planes tangent to the boundary of $C_0$. If we reflect $(1,0,0)$, i.e., the location of the delta function, through every plane tangent to the boundary of $C_0$, this forms a cardiod, and this is the reason for the cardiod shape observed in figure \ref{bolker_3D}. For another example, when the delta function is located at $\vx = (0,0,0)$, the corresponding artifact curve is a circle radius 2, center $(0,0,0)$, as shown in figure \ref{bolker_2D_b}. All artifacts due to Bolker lie outside of $C_0$, as predicted by our theory, and thus do not interfere with the scanning region.

\begin{figure}[!h]
\hspace{1.4cm}
\begin{subfigure}{0.421\textwidth}
\includegraphics[width=1.2\linewidth, height=4.5cm, keepaspectratio]{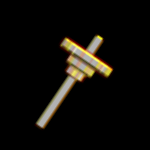} 
\subcaption*{} \label{GT_3D_a}
\end{subfigure}
\begin{subfigure}{0.421\textwidth}
\includegraphics[width=1.2\linewidth, height=4.5cm, keepaspectratio]{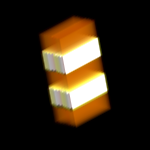}
\subcaption*{} \label{GT_3D_b}
\end{subfigure}
\begin{subfigure}{0.43\textwidth}
\includegraphics[width=1.2\linewidth, height=4.5cm, keepaspectratio]{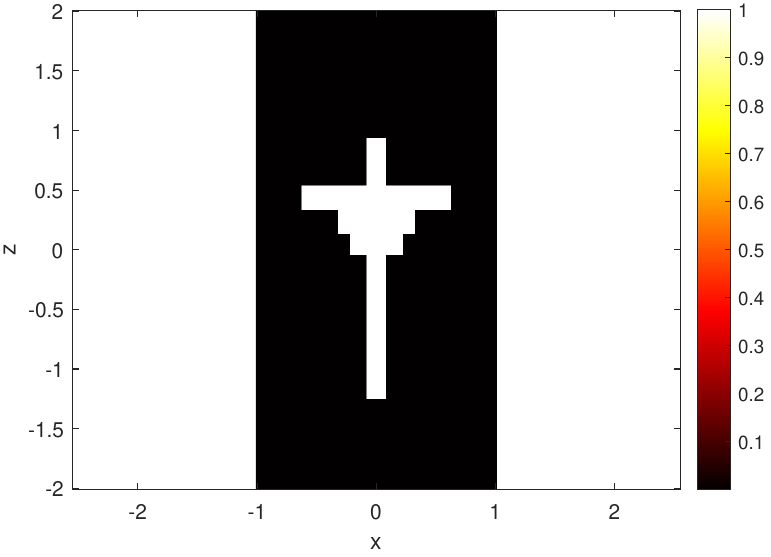} 
\subcaption*{spin top} \label{GT_2D_a}
\end{subfigure}
\begin{subfigure}{0.43\textwidth}
\includegraphics[width=1.2\linewidth, height=4.5cm, keepaspectratio]{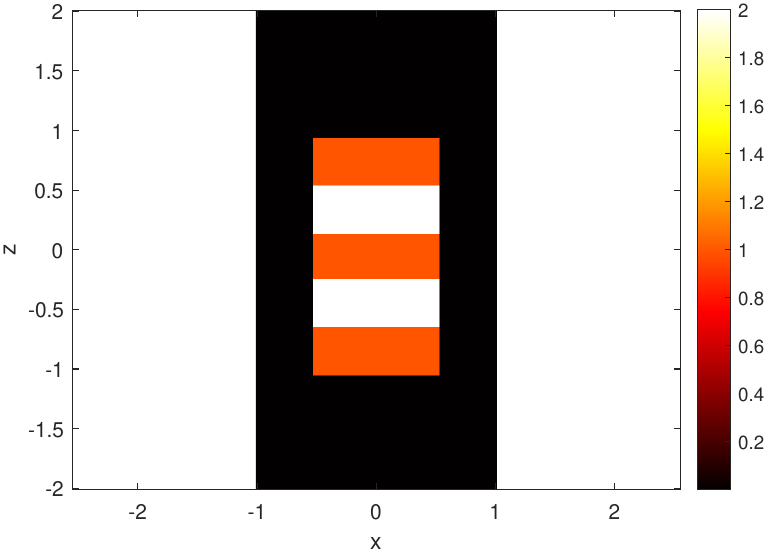}
\subcaption*{layered bricks} \label{GT_2D_b}
\end{subfigure}
\caption{Image phantoms. Top row - 3-D renderings. Bottom row - $(x,z)$ plane slices.}
\label{GT_3D}
\end{figure}
\begin{figure}[!h]
\centering
\hspace{1.4cm}
\begin{subfigure}{0.421\textwidth}
\includegraphics[width=1.2\linewidth, height=4.5cm, keepaspectratio]{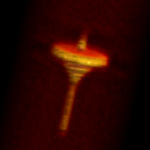} 
\subcaption*{}
\end{subfigure}
\begin{subfigure}{0.421\textwidth}
\includegraphics[width=1.2\linewidth, height=4.5cm, keepaspectratio]{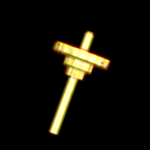}
\subcaption*{}
\end{subfigure}\\
\hspace{1.4cm}
\begin{subfigure}{0.421\textwidth}
\includegraphics[width=1.2\linewidth, height=4.5cm, keepaspectratio]{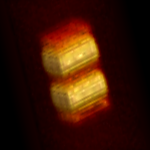} 
\subcaption*{\hspace{-2cm} Landweber}
\end{subfigure}
\begin{subfigure}{0.421\textwidth}
\includegraphics[width=1.2\linewidth, height=4.5cm, keepaspectratio]{ 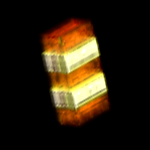}
\subcaption*{\hspace{-2cm} CGLS-TV}
\end{subfigure}
\caption{3-D image reconstructions with $0.1\%$ added Gaussian noise. Top row - spin top reconstruction. Bottom row - layered brick reconstructions.}
\label{F_r1}
\end{figure}
\begin{figure}[!h]
\centering
\begin{subfigure}{0.43\textwidth}
\includegraphics[width=1.2\linewidth, height=4.5cm, keepaspectratio]{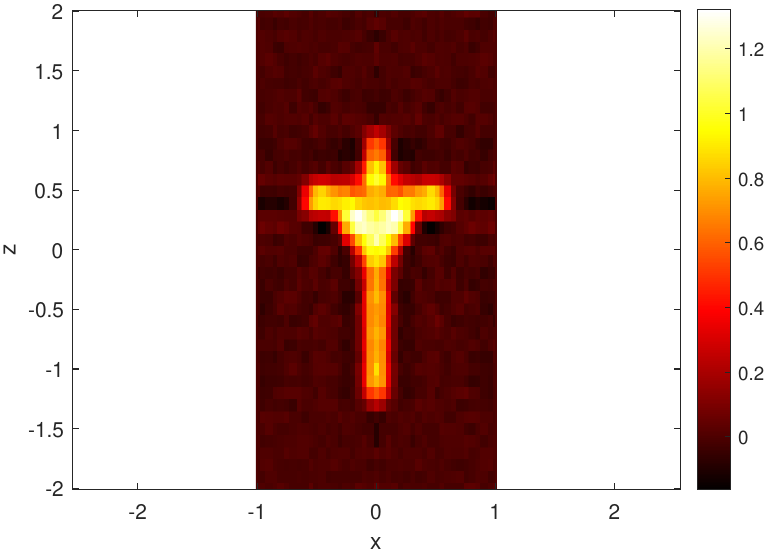} 
\end{subfigure}
\begin{subfigure}{0.43\textwidth}
\includegraphics[width=1.2\linewidth, height=4.5cm, keepaspectratio]{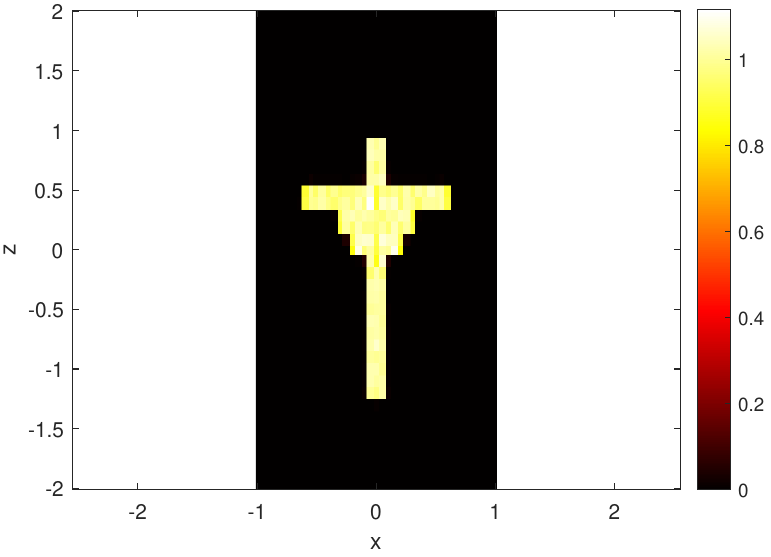}
\end{subfigure}
\begin{subfigure}{0.43\textwidth}
\includegraphics[width=1.2\linewidth, height=4.5cm, keepaspectratio]{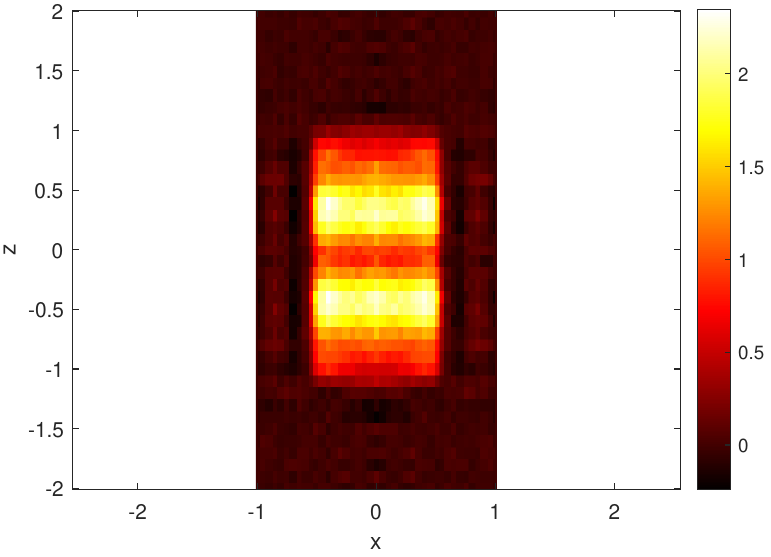} 
\subcaption*{Landweber}
\end{subfigure}
\begin{subfigure}{0.43\textwidth}
\includegraphics[width=1.2\linewidth, height=4.5cm, keepaspectratio]{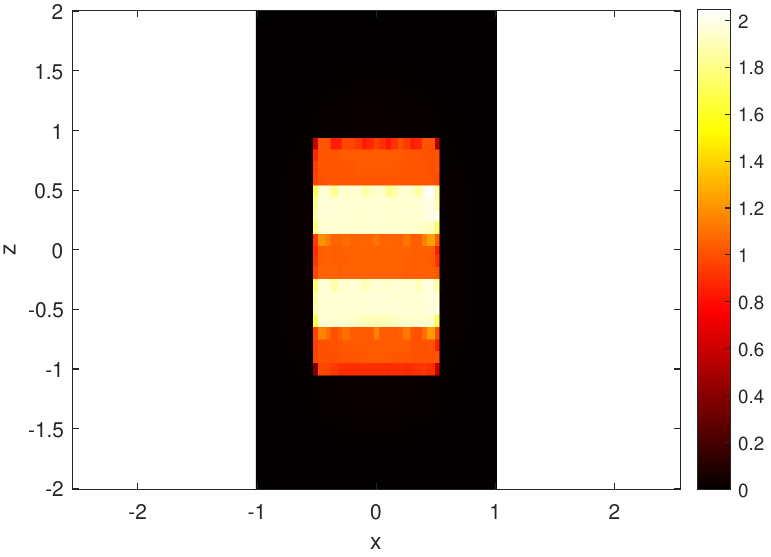}
\subcaption*{CGLS-TV}
\end{subfigure}
\caption{$(x,z)$ plane slices corresponding to the reconstructions of figure \ref{F_r1}. Top row - spin top reconstructions. Bottom row - layered brick reconstructions.}
\label{F_r2}
\end{figure}

\subsection{Phantom reconstructions}
\label{phantom_results}
In this subsection, we present simulated reconstructions of image phantoms using iterative methods, as detailed in subsection \ref{recon_mthd}. See figure \ref{GT_3D}, where we show two test phantoms. The first test phantom, shown on the left-hand of figure \ref{GT_3D}, is in the form of a spin top and is rotationally symmetric about the $z$ axis. The spin top has density 1, and the background has density 0. The second test phantom, shown on the right-hand of figure \ref{GT_3D}, is composed of layered bricks, with densities alternating between 1 and 2. The background of the layered brick phantom has density 0, as with the spin top phantom.

In figures \ref{F_r1} and \ref{F_r2}, we show reconstructions of the spin top and layered brick phantoms using the Landweber and CGLS-TV methods, when the noise level is set to $\gamma = 0.001$. In figures \ref{F_r1} and \ref{F_r2}, we show 3-D renderings  and $(x,z)$ plane cross sections of the image reconstructions. In table \ref{Tr1}, we tabulate the least squares error values, $\epsilon_r$, for varying levels of added noise $\gamma$. 

The Landweber reconstructions on the left-hand of figures \ref{F_r1} and \ref{F_r2} are blurry and corrupted by noise. This is also reflected in the least squares error values in table \ref{Tr1}. For example, the least squares error offered by Landweber on the spin top phantom is $43\%$. We see a significant blurring effect near the boundaries of the phantoms, which is particularly emphasized near edges in directions parallel to the $z$ axis. The layered brick phantom has jump discontinuities in the $x$, $y$, and $z$ directions. In the bottom left-hand subfigure of figure \ref{F_r2}, the oscillating layers in the $z$ direction are heavily blurred together, and the edges are not well resolved. The edges in the $x$ direction, while still slightly blurred, are better resolved than those in the $z$ direction. In section \ref{visible}, we showed that there are parts of the image edge map which are not visible in $\Rc_L f$ data, and that the missing edges are concentrated near directions parallel to the $z$ axis. 
%In figure \ref{Fvis_2a}, we showed that a maximum of $45\%$ of edges were detectable at the origin, and thus the problem is one of limited-angle tomography. 
This explains why we see a strong blurring effect in the reconstruction near edges with direction parallel to, or a small angle from the $z$ direction.

The CGLS-TV reconstructions are shown in the right-hand of figures \ref{F_r1} and \ref{F_r2}. CGLS-TV offers better image quality, when compared to the Landweber method, and TV suppresses the noise and blurring due to limited angles. The least squares error scores are also significantly improved. For example, the least squares error offered by CGLS-TV on the spin top phantom is $4\%$, which is $39\%$ lower than the $43\%$ error offered by the Landweber method. 
\begin{table}[!ht]
    \begin{subtable}{.49\linewidth}\centering
{ \begin{tabular}{|l|l|l|l|l|}
    \hline
        method & $\gamma = 0.001$ & $\gamma = 0.01$ & $\gamma = 0.05$ \\ \hline
        Landweber & 0.43 & 0.46 & 0.49 \\ \hline
        CGLS-TV & 0.04 & 0.12 & 0.29  \\ \hline
    \end{tabular} }
\caption{spin top}
\end{subtable}

\begin{subtable}{.49\linewidth}\centering {
\begin{tabular}{|l|l|l|l|l|}
    \hline
        method & $\gamma = 0.001$ & $\gamma = 0.01$ & $\gamma = 0.05$ \\ \hline
        Landweber & 0.28 & 0.30 & 0.38  \\ \hline
        CGLS-TV & 0.10 & 0.12 & 0.29  \\ \hline
    \end{tabular} }
\caption{layered bricks}
\end{subtable}
\caption{Least squares error, $\epsilon_r$, values corresponding to the spin top and layered bricks phantom reconstructions, for varying levels of added noise $\gamma$.}
\label{Tr1}
\end{table}

%As the noise level, $\gamma$, increases, the reconstruction error, $\epsilon_r$, increases. 
The least squares error scores, $\epsilon_r$, offered by CGLS-TV are less than those offered by Landweber over all values of $\gamma$ considered. The $\epsilon_r$ values offered by CGLS-TV are satisfactory for $\gamma\leq 0.01$, with $\epsilon_r \leq 12\%$ for $\gamma \leq 0.01$. We see a significant increase in error using CGLS-TV, up to $\epsilon_r = 29\%$, when $\gamma = 0.05$, on both image phantoms. The inverse of $\Rc_L$ is unstable, e.g., due to limited-angles as discussed in section \ref{visible}, and the system matrix, $A$, used here is underdetermined, which may help to explain the high sensitivity to noise and the increase in reconstruction error with $\gamma$. To implement the CGLS-TV and Landweber methods, $A$ was stored as a sparse matrix, which limits the size of $A$ due to memory concerns. In further work, we aim to calculate the coefficients of $A$ ``on the fly", e.g., as in \cite[section 5.4]{thompson2011source}, so that we can increase the image resolution and the number of data points without running into memory issues.

The reconstructions presented here serve as preliminary results to validate our microlocal theory, and to show that $f$ can be recovered from $\Rc_L f$ data in an idealized setting, with modest levels of added Gaussian noise. 
% In further work, we aim to 

\section{Conclusions and further work}
In this paper, we introduced a novel scanning modality in 3-D CST, and new generalized Radon transforms, $\Rc$ and $\Rc_L$, which mathematically model the Compton scatter signal. We showed that $\Rc_L$ and $\Rc$, on domain $L^2_c(C_{\epsilon})$,  for some fixed $\epsilon \in (0,1)$, are injective and FIO which satisfy the Bolker condition. In section \ref{visible}, we investigated the edge detection capabilities of $\Rc$ and $\Rc_L$, and showed that there are elements of the edge map of $f$ which are invisible in $\Rc f$ and $\Rc_L f$ data, and thus the problem is one of limited-angle tomography. 
%This work lays the theoretical foundation for three-dimensional Compton imaging using the proposed scanner design. 
In section \ref{results}, we presented simulated image reconstructions using $\Rc_L f$ data. To validate our microlocal theory, we showed reconstructions of a delta function. The predicted and observed artifacts superimposed exactly, and were located outside the region of interest, $C_{\epsilon}$, as predicted by our theory. In subsection \ref{phantom_results}, we presented reconstructions of image phantoms in an idealized setting, using iterative methods. The data was simulated using the exact model $\Rc_L f$. We then added varying levels of Gaussian noise to simulate noise. In further work, we aim to consider more accurate data generation methods, such as Monte Carlo simulation. 

The proposed cylindrical scanner has the ability measure conventional X-ray CT data (straight line integrals), in addition to Compton scatter data. In further work, we aim to combine X-ray CT and Compton scatter data to address noise and inversion instabilities due to limited angles, e.g., as discussed in section \ref{visible}.

\section*{Acknowledgements} The author wishes to acknowledge
funding support from The V
Foundation, The Brigham Ovarian Cancer Research Fund, Abcam Inc., and Aspira Women's Health. The author would also like to thank the Isaac Newton Institute for Mathematical Sciences for support and hospitality during the workshop Rich and Nonlinear Tomography while this paper was being formulated. This work was supported by EPSRC grant number EP/R014604/1.

\bibliographystyle{abbrv} 
\bibliography{RefEllipsoid_D1}

\begin{thebibliography}{10}

\bibitem{ABKQ2013}
G.~Ambartsoumian, J.~Boman, V.~P. Krishnan, and E.~T. Quinto.
\newblock Microlocal analysis of an ultrasound transform with circular source
  and receiver trajectories.
\newblock In {\em Geometric analysis and integral geometry}, volume 598 of {\em
  Contemp. Math.}, pages 45--58. Amer. Math. Soc., Providence, RI, 2013.

\bibitem{cebeiro2017new}
J.~Cebeiro, M.~K. Nguyen, M.~A. Morvidone, and A.~Noumow{\'e}.
\newblock New “improved” compton scatter tomography modality for
  investigative imaging of one-sided large objects.
\newblock {\em Inverse Problems in Science and Engineering}, 25(11):1676--1696,
  2017.

\bibitem{cebeiro2021three}
J.~Cebeiro, C.~Tarpau, M.~A. Morvidone, D.~Rubio, and M.~K. Nguyen.
\newblock On a three-dimensional compton scattering tomography system with
  fixed source.
\newblock {\em Inverse Problems}, 37(5):054001, 2021.

\bibitem{duistermaat1996fourier}
J.~J. Duistermaat and L.~Hormander.
\newblock {\em {F}ourier integral operators}, volume~2.
\newblock Springer, 1996.

\bibitem{gazzola2019ir}
S.~Gazzola, P.~C. Hansen, and J.~G. Nagy.
\newblock Ir tools: a matlab package of iterative regularization methods and
  large-scale test problems.
\newblock {\em Numerical Algorithms}, 81(3):773--811, 2019.

\bibitem{godeke2022imaging}
J.~G{\"o}deke and G.~Rigaud.
\newblock Imaging based on compton scattering: model uncertainty and
  data-driven reconstruction methods.
\newblock {\em Inverse Problems}, 2022.

\bibitem{guerrero2017three}
P.~Guerrero~Prado, M.~K. Nguyen, L.~Dumas, and S.~X. Cohen.
\newblock Three-dimensional imaging of flat natural and cultural heritage
  objects by a compton scattering modality.
\newblock {\em Journal of Electronic Imaging}, 26(1):011026--011026, 2017.

\bibitem{GS1977}
V.~Guillemin and S.~Sternberg.
\newblock {\em {Geometric Asymptotics}}.
\newblock American Mathematical Society, Providence, RI, 1977.

\bibitem{hansen2018air}
P.~C. Hansen and J.~S. J{\o}rgensen.
\newblock Air tools ii: algebraic iterative reconstruction methods, improved
  implementation.
\newblock {\em Numerical Algorithms}, 79(1):107--137, 2018.

\bibitem{hormanderI}
L.~H{\"o}rmander.
\newblock {\em The analysis of linear partial differential operators. {I}}.
\newblock Classics in Mathematics. Springer-Verlag, Berlin, 2003.
\newblock Distribution theory and {F}ourier analysis, Reprint of the second
  (1990) edition [Springer, Berlin].

\bibitem{hormanderIII}
L.~H\"{o}rmander.
\newblock {\em The analysis of linear partial differential operators. {III}}.
\newblock Classics in Mathematics. Springer, Berlin, 2007.
\newblock Pseudo-differential operators, Reprint of the 1994 edition.

\bibitem{hormander}
L.~H\"{o}rmander.
\newblock {\em The analysis of linear partial differential operators. {IV}}.
\newblock Classics in Mathematics. Springer-Verlag, Berlin, 2009.
\newblock Fourier integral operators, Reprint of the 1994 edition.

\bibitem{krishnan2015microlocal}
V.~P. Krishnan and E.~T. Quinto.
\newblock Microlocal analysis in tomography.
\newblock {\em Handbook of mathematical methods in imaging}, 1:3, 2015.

\bibitem{kuger2022multiple}
L.~Kuger and G.~Rigaud.
\newblock On multiple scattering in compton scattering tomography and its
  impact on fan-beam ct.
\newblock {\em Inverse Problems and Imaging}, 16(5):1359--1387, 2022.

\bibitem{landweber1951iteration}
L.~Landweber.
\newblock An iteration formula for fredholm integral equations of the first
  kind.
\newblock {\em American journal of mathematics}, 73(3):615--624, 1951.

\bibitem{NT}
M.~Nguyen and T.~T. Truong.
\newblock {Inversion of a new circular-arc Radon transform for {C}ompton
  scattering tomography}.
\newblock {\em Inverse Problems}, 26(6):065005, 2010.

\bibitem{norton}
S.~J. Norton.
\newblock {Compton scattering tomography}.
\newblock {\em Journal of applied physics}, 76(4):2007--2015, 1994.

\bibitem{pal}
V.~P. Palamodov.
\newblock {An analytic reconstruction for the {C}ompton scattering tomography
  in a plane}.
\newblock {\em Inverse Problems}, 27(12):125004, 2011.

\bibitem{palamodov2012uniform}
V.~P. Palamodov.
\newblock A uniform reconstruction formula in integral geometry.
\newblock {\em Inverse Problems}, 28(6):065014, 2012.

\bibitem{perraudin2014unlocbox}
N.~Perraudin, V.~Kalofolias, D.~Shuman, and P.~Vandergheynst.
\newblock Unlocbox: A matlab convex optimization toolbox for proximal-splitting
  methods.
\newblock {\em arXiv preprint arXiv:1402.0779}, 2014.

\bibitem{quinto}
E.~T. Quinto.
\newblock {The dependence of the generalized {R}adon transform on defining
  measures}.
\newblock {\em Trans. Amer. Math. Soc.}, 257:331--346, 1980.

\bibitem{redler2018compton}
G.~Redler, K.~C. Jones, A.~Templeton, D.~Bernard, J.~Turian, and J.~C. Chu.
\newblock Compton scatter imaging: A promising modality for image guidance in
  lung stereotactic body radiation therapy.
\newblock {\em Medical physics}, 45(3):1233--1240, 2018.

\bibitem{RigaudComptonSIIMS2017}
G.~Rigaud.
\newblock Compton scattering tomography: feature reconstruction and
  rotation-free modality.
\newblock {\em SIAM J. Imaging Sci.}, 10(4):2217--2249, 2017.

\bibitem{rigaud20213d}
G.~Rigaud.
\newblock 3d compton scattering imaging with multiple scattering: Analysis by
  fio and contour reconstruction.
\newblock {\em Inverse Problems}, 2021.

\bibitem{rigaud2021reconstruction}
G.~Rigaud and B.~Hahn.
\newblock Reconstruction algorithm for 3d compton scattering imaging with
  incomplete data.
\newblock {\em Inverse Problems in Science and Engineering}, 29(7):967--989,
  2021.

\bibitem{rigaud20183d}
G.~Rigaud and B.~N. Hahn.
\newblock 3{D} {C}ompton scattering imaging and contour reconstruction for a
  class of {R}adon transforms.
\newblock {\em Inverse Problems}, 34(7):075004, 2018.

\bibitem{Rudin:FA}
W.~Rudin.
\newblock {\em Functional analysis}.
\newblock McGraw-Hill Book Co., New York, 1973.
\newblock McGraw-Hill Series in Higher Mathematics.

\bibitem{schiefeneder2017radon}
D.~Schiefeneder and M.~Haltmeier.
\newblock The radon transform over cones with vertices on the sphere and
  orthogonal axes.
\newblock {\em SIAM Journal on Applied Mathematics}, 77(4):1335--1351, 2017.

\bibitem{tarpau2020analytic}
C.~Tarpau, J.~Cebeiro, M.~K. Nguyen, G.~Rollet, and M.~A. Morvidone.
\newblock Analytic inversion of a radon transform on double circular arcs with
  applications in compton scattering tomography.
\newblock {\em IEEE Transactions on Computational Imaging}, 6:958--967, 2020.

\bibitem{thompson2011source}
W.~M. Thompson.
\newblock {\em Source Firing Patterns and Reconstruction Algorithms for a
  Switched Source, O set Detector CT Machine}.
\newblock PhD thesis, The University of Manchester (United Kingdom), 2011.

\bibitem{tricomi1985integral}
F.~G. Tricomi.
\newblock {\em Integral equations}, volume~5.
\newblock Courier corporation, 1985.

\bibitem{truong2019compton}
T.-T. Truong and M.~K. Nguyen.
\newblock Compton scatter tomography in annular domains.
\newblock {\em Inverse Problems}, 2019.

\bibitem{me}
J.~Webber.
\newblock {X-ray {C}ompton scattering tomography}.
\newblock {\em Inverse problems in science and engineering}, 24(8):1323--1346,
  2016.

\bibitem{webber2020compton}
J.~Webber and E.~L. Miller.
\newblock Compton scattering tomography in translational geometries.
\newblock {\em Inverse Problems}, 36(2):025007, 2020.

\bibitem{webber2022ellipsoidal}
J.~W. Webber, S.~Holman, and E.~T. Quinto.
\newblock Ellipsoidal and hyperbolic radon transforms; microlocal properties
  and injectivity.
\newblock {\em arXiv preprint arXiv:2212.00243}, 2022.

\bibitem{me2}
J.~W. Webber and W.~R. Lionheart.
\newblock {Three dimensional {C}ompton scattering tomography}.
\newblock {\em Inverse Problems}, 34(8):084001, 2018.

\end{thebibliography}

\end{document}